\documentclass[12pt,a4paper]{amsart}

\usepackage[left=25truemm,right=25truemm,top=30truemm,bottom=30truemm]{geometry}

\usepackage{mathtools}
\usepackage{amssymb}
\usepackage{thmtools}
\usepackage{mathrsfs}

\usepackage[dvipdfmx]{graphicx}
\usepackage{subcaption}
\usepackage{float}
\usepackage[dvipsnames]{xcolor}
\usepackage{framed}
\usepackage[all]{xy}
\usepackage{listings}
\usepackage{setspace}
\usepackage[english]{babel}
\usepackage{amscd}
\usepackage{enumitem}

\usepackage{tikz}
\usetikzlibrary{trees}

\usepackage{hyperref}
\usepackage{cleveref}

\usepackage{amsthm}

\makeatletter

\renewcommand{\section}{%
  \@startsection{section}{1}%
    {\z@}%
    {3ex}%
    {1ex}%
    {\normalfont\scshape\centering}%
}

\renewcommand{\subsection}{%
  \@startsection{subsection}{2}%
    {\z@}%
    {2ex}%
    {-.5em}%
    {\normalfont\bfseries}%
}

\makeatother

\newtheoremstyle{myplain}
  {10pt} 
  {12pt}
  {\itshape}                   
  {}                           
  {\bfseries}                  
  {.}                          
  {8pt}     
  {}

\newtheoremstyle{mydefinition}
  {10pt} 
  {12pt}
  {\normalfont}
  {}
  {\bfseries}
  {.}
  {8pt}
  {}

\theoremstyle{myplain}
\newtheorem{thm}{Theorem}[section]
\newtheorem{lem}[thm]{Lemma}
\newtheorem{cor}[thm]{Corollary}
\newtheorem{prop}[thm]{Proposition}

\theoremstyle{mydefinition}
\newtheorem{example}[thm]{Example}
\newtheorem{defi}[thm]{Definition}
\newtheorem{rem}[thm]{Remark}

\theoremstyle{myplain}
\newtheorem*{introthm}{Theorem~\ref{Thm4.3}}
\newtheorem*{introthm2}{Theorem~\ref{Thm4.8}}
\newtheorem*{introthm3}{Theorem~\ref{Thm5.12}}
\newtheorem*{introcor}{Corollary~\ref{Cor5.18}}
\newtheorem*{introthm4}{Theorem~\ref{Thm7.7}}

\usepackage{enumitem}

\newlist{examplelist}{enumerate}{1}

\setlist[examplelist,1]{
  label=\textup{(\arabic*)},
  ref=(\arabic*),
  leftmargin=2.25em,
  labelsep=.6em,
  topsep=.3\baselineskip,
  itemsep=.2\baselineskip,
  parsep=0pt,
  partopsep=0pt
}

\AfterEndEnvironment{proof}{\par\addvspace{1\baselineskip}}

\DeclareMathOperator{\Int}{Int}
\DeclareMathOperator{\rank}{rank}

\Crefname{figure}{Figure}{Figures}
\numberwithin{equation}{section}

\newcommand{\id}{{\rm{id}}}

\title{Gluing $\partial$-Morin maps on manifolds with boundary \\and its applications}

\author[K. Iwakura]{Koki Iwakura}
\address{Joint Graduate School of Mathematics for Innovation, Kyushu University, Motooka 744, Nishiku, Fukuoka 819-0395, Japan.}
\email{iwakura.koki0105@gmail.com}

\date{}

\begin{document}
\maketitle

\begin{abstract}
We study $\partial$-Morin maps, a class of smooth maps from manifolds with boundary such that both the maps themselves and their restrictions to the boundary have only Morin singular points, and the singular point sets of the maps are disjoint from the boundary.
We develop a gluing construction that provides a unified framework for studying maps from manifolds with boundary through corresponding maps from closed manifolds. 
Using this framework, we derive Euler characteristic formulas for $\partial$-Morin maps and a congruence relating the signature to the self-intersection number of the singular point set for $\partial$-fold maps from $4$-manifolds to $3$-manifolds. 
We also establish an inequality relating singular fibers of stable maps from a $3$-manifold to the plane to the simplicial volume of the source manifold.
We further apply these results to the existence problem for $\partial$-fold maps and to the non-singular extension problem.
\end{abstract}

\vspace{-5pt}

\tableofcontents

\section{Introduction}

\subsection{Gluing construction for maps on manifolds with boundary}
The \emph{singular points} of a smooth map are the points of the source manifold at which its differential fails to have maximal rank.
Although this is a local condition, the singular points of a map can impose constraints on the global structures such as topology and smooth structure of the source manifold~\cite{OSS, Sae2, SS}.
However, previous research has focused mainly on maps from closed manifolds, with comparatively less attention paid to maps from manifolds with boundary.
The presence of a boundary introduces additional phenomena that do not arise in the closed case: one must simultaneously consider the singular points of the map itself and those of its restriction to the boundary.
To address these difficulties, various methods tailored to specific settings have been developed~\cite{BNR, Haj, Iwa3, Iwa4, JR, SY2, Shi}.
However, a general method applicable to broader classes of maps has yet to be established.
In this paper, we take a different approach: we relate maps from manifolds with boundary to maps from closed manifolds.

To this end, we build on the gluing construction introduced by Levine~\cite{Lev2}, which produces a map from a closed manifold by gluing together two maps from manifolds with boundary whose restrictions to the boundary are compatible under the gluing. 
We apply this construction to \emph{$\partial$-Morin maps}, a broad class of maps from manifolds with boundary, and obtain several results relating their singular points to global properties of their source manifolds.
In this way, the gluing construction provides a common framework for studying maps from manifolds with boundary while allowing us to draw on results established in the closed setting.
The key point is that the singular points of the restriction to the boundary give rise, under the gluing construction, to ordinary Morin singular points of the resulting map on a closed manifold.

\subsection{Relative formulas for $\partial$-Morin maps}
\emph{Morin singular points} form a fundamental class of singular points, generalizing Morse critical points to maps with higher-dimensional targets.
Smooth maps having only Morin singular points are called \emph{Morin maps}; this class includes a wide variety of maps on closed manifolds, including Morse functions.
In this paper, we introduce $\partial$-Morin maps as a counterpart of Morin maps for manifolds with boundary.
A $\partial$-Morin map is a smooth map from a manifold with boundary such that both the map itself and its restriction to the boundary have only Morin singular points, and the singular point set of the map is disjoint from the boundary of the source manifold.
This class encompasses several previously studied classes, including $m$-functions~\cite{Haj, JR}, boundary special generic maps~\cite{Iwa3, Shi}, and submersions with definite folds~\cite{Iwa4}.

For $\partial$-Morin maps, we establish relative versions of Fukuda-type theorems~\cite{DF, Fuk, Sae3, Rui} and Sakuma's theorem~\cite{Sak}.
Fukuda-type theorems express the Euler characteristic of the source manifold in terms of the stratification of the singular point set according to the type of Morin singular points.
For Morse functions, these theorems recover the classical formula expressing the Euler characteristic as the alternating sum of the numbers of critical points of each index~\cite{Mil}.
In this paper, we obtain the following relative formula by incorporating contributions from the singular points of the restriction to the boundary.

\begin{introthm}\label{Thm4.3}
Let $F\colon N\to\mathbb{R}^l$ be a $\partial$-Morin map of a compact $n$-dimensional manifold $N$ with boundary, where $n>l$. 
If $n$ is even and $l$ is odd, then we have
$$
\chi(N)
=
\sum_{k:\vspace{1pt}\text{odd}}\Big(\chi(\overline{A_k^+(F)})-\chi(\overline{A_k^-(F)})\ \Big)
+
\frac{1}{2}\sum_{k:\vspace{1pt}\text{odd}}\Big(\chi(\overline{{}^{\partial\!}A_k^+(F)})-\chi(\overline{{}^{\partial\!}A_k^-(F)})\Big). 
$$
Here, $\chi$ is the Euler characteristic of the relevant set, and the sets $A_k^\pm(F)$ and ${}^{\partial\!}A_k^\pm(F)$ are the submanifolds of $N$ defined in Section~4.
\end{introthm}

We obtain this formula by combining the gluing construction with a Fukuda-type theorem of Ruiz~\cite{Rui}.
Moreover, under suitable assumptions on the dimensions and the types of singular points, the same construction, together with the results of Saeki~\cite{Sae3} and Dutertre--Fukui~\cite{DF}, yields analogous formulas for maps into general target manifolds in \Cref{Thm4.5} and \Cref{Thm4.6}.

We next establish a relative version of Sakuma's theorem.
A \emph{$\partial$-fold map} is a $\partial$-Morin map whose singular points and those of the restriction to the boundary are all fold points, which form a special class of Morin singular points.
For stable maps from closed oriented $4$-manifolds to $\mathbb{R}^3$, Sakuma~\cite{Sak} proved a congruence modulo $4$ relating the signature of the source manifold to the self-intersection number of the singular point set.
For $\partial$-fold maps from $4$-manifolds with boundary to $3$-manifolds, we obtain the following congruence.

\begin{introthm2}\label{Thm4.8}
Let $N$ be a compact oriented $4$-manifold whose boundary $\partial N$ is an integral homology sphere and let $L$ be a connected parallelizable $3$-manifold. 
If there exists a $\partial$-fold map $F\colon N\to L$, then 
$$
\sigma(N)+S(F)\cdot S(F)\equiv 2+\chi({}^{\partial\!}A_1^+(F))-\chi({}^{\partial\!}A_1^-(F)) \pmod 4, 
$$
where $\sigma(N)$ is the signature of $N$, $S(F)\cdot S(F)$ is the self-intersection number of $S(F)$ in $N$, and ${}^{\partial\!}A_1^{\pm}(F)$ are the same as those in \Cref{Thm4.5}.
\end{introthm2}

We prove this theorem by combining \Cref{Thm4.5} with a formula due to Klug~\cite{Klu}.
Moreover, when the target is $\mathbb{R}^3$, the submersion on a collar neighborhood of the boundary induces a stable framing on the boundary.
The contribution to the right-hand side of the congruence in \Cref{Thm4.8} from the singular points of the restriction to the boundary admits a geometric interpretation in terms of the Hirzebruch defect of this framing.

\subsection{Singular fibers and the topology of $3$-manifolds with boundary}
A smooth map is called \emph{stable} if every sufficiently small perturbation is equivalent to it under diffeomorphisms of the source and target.
For such a map, the germ of the map along the inverse image of a singular value is called a \emph{singular fiber}.
For stable maps from closed manifolds, the types and numbers of singular fibers are known to reflect the topology of the source manifold~\cite{IK, KLP, SY1, SY2}.
Ishikawa--Koda~\cite{IK} introduced the \emph{stable map complexity} of a compact orientable $3$-manifold whose boundary is either empty or a disjoint union of tori, in terms of a weighted count of singular fibers of types ${\rm II}^2$ and ${\rm II}^3$ occurring in a class of maps called \emph{$S$-maps}.
They further established an inequality bounding the simplicial volume of the source manifold in terms of this complexity.
However, the class of $S$-maps does not include stable maps from manifolds with boundary.

In this paper, we extend the notion of stable map complexity and the associated inequality to the setting of stable maps from compact $3$-manifolds with boundary to the plane.
More precisely, for a compact $3$-manifold $N$ with boundary, we define its \emph{$\partial$-stable map complexity} $\mathrm{smc}^{\partial}(N)$ in terms of singular fibers of types $\mathrm{bII}^{20}$ and $\mathrm{bII}^{21}$ occurring in stable maps from $N$ to the plane.
Then, we establish the following inequality relating this complexity to the simplicial volume of $N$.

\begin{introthm3}
Let $N$ be a compact connected oriented $3$-manifold with toroidal boundary. 
Then, 
$$
\lVert N,\partial N\rVert V_{\text{tet}}\leq {\rm smc}^{\partial}(N)V_{\text{oct}}, 
$$
where $\lVert N,\partial N\rVert$ is the simplicial volume of $N$, and $V_{\text{tet}}=1.01…$ and $V_{\text{oct}}=3.66…$ denote the volumes of the ideal regular tetrahedron and the ideal regular octahedron, respectively. 
\end{introthm3}

To prove this inequality, we first double the map using the gluing construction and determine how singular fibers of types $\mathrm{bII}^{20}$ and $\mathrm{bII}^{21}$ give rise to singular fibers of types $\mathrm{II}^2$ and $\mathrm{II}^3$ in the resulting map.
We then apply the inequality of Ishikawa--Koda to the resulting stable map from a closed manifold.
As an application, we obtain lower bounds on the number of singular fibers occurring in stable maps from the exteriors of hyperbolic knots to the plane.

We further introduce \emph{simple stable maps}, a class of stable maps from manifolds with boundary defined by a condition on their fibers.
Such maps have no singular fibers of types $\mathrm{bII}^{20}$ or $\mathrm{bII}^{21}$.
Recall that the simplicial volume of a compact orientable irreducible $3$-manifold with empty or toroidal boundary is proportional to the total hyperbolic volume of the hyperbolic pieces in its JSJ decomposition.
Combining this observation with the inequality above, we obtain the following result relating the existence of simple stable maps to graph manifolds.

\begin{introcor}
Let $N$ be a compact connected irreducible oriented $3$-manifold with toroidal boundary. 
If $N$ admits a simple stable $\partial$-fold map into $\mathbb{R}^2$ without singular points, then $N$ is a graph manifold. 
\end{introcor}

Saeki~\cite{Sae4} proved that a closed orientable $3$-manifold admits a simple stable map to the plane if and only if it is a graph manifold.
\Cref{Cor5.18} provides a relative analogue of one direction of this characterization.

\subsection{Application to the existence problem for $\partial$-fold maps}
A smooth map whose singular points are all fold points is called a \emph{fold map}; this class of maps generalizes Morse functions.
Although every closed smooth manifold admits a Morse function, a fold map to a prescribed higher-dimensional target need not exist.
This leads to the following question: when does a given closed manifold admit a fold map to a prescribed target, such as a Euclidean space of a fixed dimension?
This existence problem is fundamental to understanding the relationship between fold points and the global structure of closed manifolds.
Such problems have been extensively studied; see, for example, \cite{And, Eli, SSS}.

In this paper, we study the existence problem for $\partial$-fold maps as an analogue of the problem above.
When the target is $\mathbb{R}$, a $\partial$-fold map is precisely an $m$-function, and every compact manifold with boundary admits such a function~\cite{JR}.
However, for targets of dimension at least $2$, the existence problem remains open in general.
Using the gluing construction together with the relative results established above, we obtain both existence and non-existence results for $\partial$-fold maps.
As a family of test cases, we consider punctured complex projective spaces $\mathbb{C}P^m\setminus\Int D^{2m}$ and investigate the existence of $\partial$-fold maps from these manifolds to $\mathbb{R}^l$.
For $1\leq m,l\leq 6$, \Cref{Fig1} summarizes the existence and non-existence results and indicates the cases not settled by our methods.
The case $m=2$ and $l=3$, which was left open in earlier work~\cite{Iwa3}, remains unresolved.
Nevertheless, by applying the relative version of Sakuma's theorem in \Cref{Thm4.8}, we obtain a necessary condition for the existence of a $\partial$-fold map $\mathbb{C}P^2\setminus\Int D^4\to\mathbb{R}^3$.
This condition shows that the existence problem depends not only on the global properties of the source manifold but also on the behavior of the map on a collar neighborhood of its boundary.
A similar dependence on boundary data appears in earlier studies of the existence problems for boundary special generic maps~\cite{Iwa3} and submersions with definite folds~\cite{Iwa4}, both of which are special classes of $\partial$-fold maps.

\begin{figure}[t]
\centering
\begin{tikzpicture}[scale=0.85]

\draw[step=1, gray!35, thin] (0.5,0.5) grid (6.5,6.5);

\draw[->, thick] (0.5,0.5) -- (6.7,0.5)
  node[right] {$\ell$};
\draw[->, thick] (0.5,0.5) -- (0.5,6.7)
  node[above] {$m$};

\foreach \x in {1,...,6}
  \node[below] at (\x,0.5) {\(\x\)};

\foreach \y in {1,...,6}
  \node[left] at (0.5,\y) {\(\y\)};

%l=1
\draw[blue!70!black, line width=1pt] (1,1) circle[radius=0.25];
\draw[blue!70!black, line width=1pt] (1,2) circle[radius=0.25];
\draw[blue!70!black, line width=1pt] (1,3) circle[radius=0.25];
\draw[blue!70!black, line width=1pt] (1,4) circle[radius=0.25];
\draw[blue!70!black, line width=1pt] (1,5) circle[radius=0.25];
\draw[blue!70!black, line width=1pt] (1,6) circle[radius=0.25];

%l=2
\draw[gray!100, line width=1.5pt] (1.82,1) -- (2.18,1);
\node[green!50!black] at (2,2) {\Large$\triangle$};
\draw[blue!70!black, line width=1pt] (2,3) circle[radius=0.25];
\node[green!50!black] at (2,4) {\Large$\triangle$};
\draw[blue!70!black, line width=1pt] (2,5) circle[radius=0.25];
\node[green!50!black] at (2,6) {\Large$\triangle$};

%l=3
\draw[gray!100, line width=1.5pt] (2.82,1) -- (3.18,1);
\node[green!50!black] at (3,2) {\Large$\triangle$};
\draw[blue!70!black, line width=1pt] (3,3) circle[radius=0.25];
\draw[blue!70!black, line width=1pt] (3,4) circle[radius=0.25];
\draw[blue!70!black, line width=1pt] (3,5) circle[radius=0.25];
\draw[blue!70!black, line width=1pt] (3,6) circle[radius=0.25];

%l=4
\draw[gray!100, line width=1.5pt] (3.82,1) -- (4.18,1);
\draw[gray!100, line width=1.5pt] (3.82,2) -- (4.18,2);
\draw[blue!70!black, line width=1pt] (4,3) circle[radius=0.25];
\node[green!50!black] at (4,4) {\Large$\triangle$};
\node[red!75!black]   at (4,5) {\Large$\times$};
\node[red!75!black]   at (4,6) {\Large$\times$};

%l=5
\draw[gray!100, line width=1.5pt] (4.82,1) -- (5.18,1);
\draw[gray!100, line width=1.5pt] (4.82,2) -- (5.18,2);
\draw[blue!70!black, line width=1pt] (5,3) circle[radius=0.25];
\node[green!50!black] at (5,4) {\Large$\triangle$};
\node[green!50!black] at (5,5) {\Large$\triangle$};
\node[green!50!black] at (5,6) {\Large$\triangle$};

%l=6
\draw[gray!100, line width=1.5pt] (5.82,1) -- (6.18,1);
\draw[gray!100, line width=1.5pt] (5.82,2) -- (6.18,2);
\draw[gray!100, line width=1.5pt] (5.82,3) -- (6.18,3);
\node[green!50!black] at (6,4) {\Large$\triangle$};
\node[red!75!black]   at (6,5) {\Large$\times$};
\node[red!75!black]   at (6,6) {\Large$\times$};

\end{tikzpicture}

\caption{
For $\partial$-fold maps from $\mathbb{C}P^{m}\setminus\Int D^{2m}$ to $\mathbb{R}^l$. 
A circle, a cross, and a triangle indicate existence,
non-existence, and an open case, respectively.
A thick dash indicates a case not considered in this paper.
}
\label{Fig1}
\end{figure}
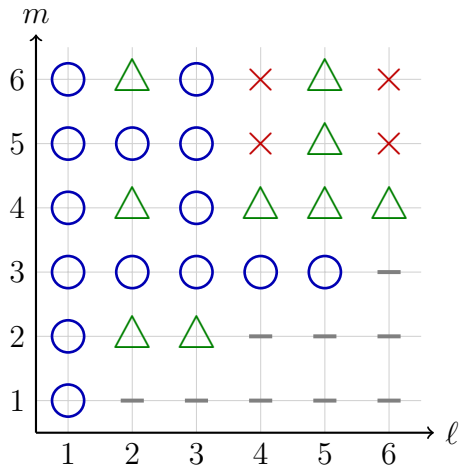

\subsection{Applications to the non-singular extension problem}
The \emph{non-singular extension problem} asks whether a map from a closed manifold extends to a submersion from a compact manifold whose boundary is the given source manifold.
The existence of such an extension is closely related to the global properties of the original map.
Therefore, this problem has long been studied as a fundamental question in the global theory of smooth maps on closed manifolds.
Blank--Laudenbach~\cite{BL} studied this problem for Morse functions on $1$-dimensional manifolds, and subsequent work treated Morse functions on higher-dimensional manifolds~\cite{Cur, Sei}.
More general maps have also been studied, but most existing results impose restrictions on the dimensions of the source and target manifolds or on the types of singular points allowed~\cite{Iwa1, Iwa2, Iwa3, Iwa4}.
A central difficulty lies in constructing an extension with no singular points while keeping its restriction to the boundary fixed.

In this paper, we consider a slightly more refined formulation of the non-singular extension problem: given a submersion defined on a collar neighborhood of the boundary whose restriction to the boundary is the prescribed map, we ask whether this submersion extends to the whole manifold.
In particular, the existence of a non-singular extension in this refined sense implies the existence of a non-singular extension in the sense described above.
This formulation appears naturally in several previous studies~\cite{BL, Cur, Sei, Iwa1, Iwa2}.
Although various obstructions to the existence of non-singular extensions have been obtained, many of them are not straightforward to verify in concrete examples.
Using the relative Fukuda-type theorems established in this paper, we derive readily computable necessary conditions for the existence of non-singular extensions.
In particular, we obtain such conditions for Morse functions and for fold maps whose singular point sets are homeomorphic to finite disjoint unions of standard spheres, and illustrate their applications through explicit examples.
We further study non-singular extensions of Morin maps from closed oriented $3$-manifolds to $\mathbb{R}^3$.
Using the gluing construction, we define an invariant associated with an extension of the prescribed map and use it to compare non-singular extensions with extensions that are not necessarily non-singular.
This leads to the following result, which shows that the existence of a non-singular extension imposes restrictions on all other extensions of the same prescribed map on a collar neighborhood of the boundary.

\begin{introthm4}\label{Thm7.7}
Let $g\colon M\times[0,1)\to L$ be a submersion such that $g|_{M\times\{0\}}$ is a Morin map from a closed oriented $3$-manifold $M\times\{0\}$ to an orientable $3$-manifold $L$. 
Suppose that $g$ admits a non-singular extension $F\colon N\to L$.
Then, for every $\partial$-Morin map $G\colon\widetilde{N}\to L$ as an extension of $g$, we have 
$$
3(\sigma(\widetilde{N})-\sigma(N))=S(G)\cdot S(G).
$$
In particular, if $\sigma(N)=\sigma(\widetilde{N})$, then $S(G)\cdot S(G)=0$.
\end{introthm4}

The proof proceeds by comparing the invariants associated with a non-singular extension and an arbitrary extension of the same prescribed map near the collar neighborhood.
Thus, the theorem exhibits a relation between different extensions of a fixed map near the boundary, rather than merely giving an obstruction to a single extension.
As an application, we construct prescribed submersions near the boundary that admit no non-singular extension to any compact oriented $4$-manifold.

\subsection*{Organization}
The paper is organized as follows. 
Section~2 reviews the basic terminology and results on singularity theory of smooth maps used throughout the paper. 
Section~3 introduces the gluing construction and establishes its basic properties. 
In Section~4, we use this construction to prove relative versions of the theorems of Fukuda and Sakuma and present several related results. 
Section~5 establishes a relative version of the theorem of Ishikawa--Koda and gives its applications. 
Section~6 applies the preceding results to the existence problem for $\partial$-fold maps. Finally, 
Section~7 presents applications to the non-singular extension problem.

\subsection*{Notation and conventions}
Unless otherwise specified, all manifolds and maps between them are assumed to be of class $C^\infty$. 
The boundary of an oriented manifold with boundary is endowed with the orientation induced by the outward-first convention.

\subsection*{Acknowledgements}
The author would like to thank Osamu Saeki and Noriyuki Hamada for their valuable comments. 
The author is also grateful to Rustam Sadykov for stimulating discussions during the author's stay at Kansas State University, which contributed to the initial development of the ideas pursued in this paper. 
The author further thanks Masato Tanabe for helpful discussions concerning Levine's work. 
This work was supported by the WISE program (MEXT) at Kyushu University and by JSPS KAKENHI Grant number JP26KJ1778.

\section{Preliminaries}
In this section, we first review some basic terminology and facts from singularity theory of maps.
After that, we introduce $\partial$-Morin maps and their basic properties.
For general background and terminology in singularity theory of maps, we refer the reader to~\cite{GG}.

\subsection{Singular points of maps}
In this subsection, we review the definition of singular points of maps and introduce the notation associated with them. 
Throughout the following three subsections, let $M$ be a compact $m$-dimensional manifold, let $L$ be an $l$-dimensional manifold, and let $f\colon M\to L$ be a map.

\begin{defi}
A point $p\in M$ is called a \emph{singular point} of $f$ if $\rank df_p<\min\{m,l\}$. 
Otherwise, $p$ is called a \emph{regular point} of $f$. 
A point $q\in L$ is called a \emph{singular value} of $f$ if there exists a singular point $p\in M$ of $f$ such that $f(p)=q$. 
Otherwise, $q$ is called a \emph{regular value} of $f$. 
In particular, when all points of $M$ are regular points of $f$, $f$ is called an \emph{immersion} if $m<l$, or $f$ is called a \emph{submersion} if $m\geq l$. 
\end{defi}

The subset of $M$ consisting of the singular points of $f$ is denoted by $S(f)$; that is, 
$$
S(f)=\{p\in M\mid \rank df_p<\min\{m,l\}\}. 
$$
This is called the \emph{singular point set} of $f$, and the image $f(S(f))\subset L$ is called the \emph{singular value set} of $f$.

\subsection{Stable maps}
In this subsection, we recall the definition of stable maps and some of their basic properties.
We first recall the definition of $C^\infty$ right-left equivalence.

\begin{defi}
Let $M_i$ be a compact $m$-dimensional manifold, let $L_i$ be an $l$-dimensional manifold, and let $f_i\colon M_i\to L_i$ be maps for $i=1,2$, where $m\geq l$.
We say that $f_1$ and $f_2$ are \emph{$C^\infty$ right-left equivalent} if there exist diffeomorphisms $\theta\colon M_1\to M_2$ and $\eta\colon L_1\to L_2$ such that the following diagram commutes:
$$
\begin{xy}
\xymatrix{
M_1 \ar[r]^{\theta} \ar[d]_{f_1}
& M_2 \ar[d]^{f_2} \\
L_1 \ar[r]_{\eta}
& L_2.
}
\end{xy}
$$
\end{defi}

\begin{rem}
By replacing the diffeomorphisms of the source and target manifolds in the above definition with homeomorphisms, one obtains the notion of $C^0$ right-left equivalence.
However, since we only use $C^\infty$ right-left equivalence in this paper, right-left equivalence always means $C^\infty$ right-left equivalence unless otherwise stated.
\end{rem}

Let $C^\infty(M,L)$ denote the space of smooth maps from $M$ to $L$ equipped with the Whitney $C^\infty$ topology.
For the definition of the Whitney $C^\infty$ topology, see~\cite{GG}.

\begin{defi}
A map $f\in C^\infty(M,L)$ is \emph{stable} if there exists an open neighborhood $U_f$ of $f$ in $C^\infty(M,L)$ such that, for every $g\in U_f$, the maps $f$ and $g$ are right-left equivalent.
Such a map $f$ is called a \emph{stable map}.
\end{defi}

Let $S^\infty(M,L)$ denote the set of stable maps from $M$ to $L$.
By Mather's results, if the dimension pair $(m,l)$ belongs to the nice dimensions, then $S^\infty(M,L)$ is dense in $C^\infty(M,L)$.
For the definition of the nice dimensions, see~\cite{Mat6}.

\subsection{Morin singular points}
In this subsection, we recall the definition of Morin singular points, which play a central role in this paper.
For the setting of this paper, we assume that $m\geq l$.

\begin{defi}\label{Def2.5}
A point $p\in M$ is called a \emph{Morin singular point} of $f$ of type $A_k$, where $1\leq k\leq l$, if there exist local coordinates $(x_1,\ldots,x_m)$ of $M$ around $p$ and $(y_1,\ldots,y_l)$ of $L$ around $f(p)$ such that $f$ is written as follows:
$$
y_i\circ f=x_i, \quad 1\leq i\leq l-1,
$$
and, for $0\leq \lambda\leq m-l$, 
$$
\begin{aligned}
y_l\circ f
&=
x_l^{k+1}
+\sum_{r=1}^{k-1}x_r x_l^{k-r}
-\sum_{r=l+1}^{l+\lambda}x_r^2
+\sum_{r=l+\lambda+1}^{m}x_r^2. 
\end{aligned}
$$
A Morin singular point of type $A_1$ is called a \emph{fold point}, and one of type $A_2$ is called a \emph{cusp point}. 
A fold point is said to be \emph{definite} if $\lambda=0$; otherwise, it is said to be \emph{indefinite}.
\end{defi}

For each $k$, we denote by $A_k(f)\subset M$ the set of Morin singular points of $f$ of type $A_k$.

\subsection{Morin maps on closed manifolds}
In this subsection, we define Morin maps on closed manifolds and recall some of their basic properties.
Let $M$ be a closed $m$-dimensional manifold, let $L$ be an $l$-dimensional manifold, and let $f\colon M\to L$ be a map.
We assume that $m\geq l$.
We first recall the definition of Morin maps.

\begin{defi}
If every singular point of $f$ is a Morin singular point, then $f$ is called a \emph{Morin map}. 
In particular, if every singular point of $f$ is a Morin singular point of type $A_1$, then $f$ is called a \emph{fold map}. 
\end{defi}

Such maps appear in many previous studies on maps defined on closed manifolds.
We recall some of them here.

\begin{example}
A function on a closed manifold is called a \emph{Morse function} if all its singular points are non-degenerate and its singular values are distinct.
By the Morse lemma~\cite{Mil}, every singular point of this map is a Morin singular point of type $A_1$, in particular, a Morse function is a fold map.
Note that a function on a closed manifold is stable if and only if it is a Morse function.
\end{example}

\begin{example}
A stable map from a closed $m$-dimensional manifold to $\mathbb{R}^2$ is a Morin map whose singular points are only of types $A_1$ and $A_2$.
Furthermore, a stable map from a closed $m$-dimensional manifold to $\mathbb{R}^3$ is a Morin map whose singular points are only of types $A_1$, $A_2$, and $A_3$.
\end{example}

\begin{rem}
When the dimension of the target is greater than three, stable maps may have singular points which are not Morin singular points.
For example, it is known that a stable map from a closed $5$-dimensional manifold to $\mathbb{R}^4$ may have singular points of type $D_4$, which are not Morin singular points, in addition to singular points of types $A_1$, $A_2$, $A_3$, and $A_4$.
See, for example, \cite[Proposition~4.1]{SY1}.
\end{rem}

The following properties of Morin maps are known.
For proofs and further properties, see~\cite{Fuk, Sae3}.

\begin{prop}
Let $f\colon M\to L$ be a Morin map. 
Then, the following hold:
\begin{enumerate}[
  label=\normalfont(\arabic*),
  leftmargin=*,
  labelsep=.5em,
  itemsep=.3em,
  topsep=.45em,
  parsep=0pt, 
  leftmargin=2em
]
\item $A_k(f)$ and $\overline{A_k(f)}$ are $(l-k)$-dimensional submanifolds of $M$.
\item $\overline{A_k(f)}=\bigcup_{i\geq k} A_i(f)$.
\item The map $f|_{A_k(f)}\colon A_k(f)\to L$ is an immersion.
\end{enumerate}
Here, $1\leq k\leq l$ and $\overline{A_k(f)}$ denotes the topological closure of $A_k(f)$.
\end{prop}

\subsection{$\partial$-Morin maps on manifolds with boundary}
In this subsection, we define $\partial$-Morin maps on manifolds with boundary and give some of their basic properties.
These maps are relative versions of the Morin maps introduced in the previous subsection.
Let $N$ be a compact $n$-dimensional manifold with boundary, let $L$ be an $l$-dimensional manifold, and let $F\colon N\to L$ be a map.
We assume that $n>l$.

\begin{defi}
If every singular point of $F$ and $F|_{\partial N}$ is a Morin singular point, and $S(F)\cap\partial N=\emptyset$, then $F$ is called a \emph{$\partial$-Morin map}.
In particular, if every singular point of $F$ and $F|_{\partial N}$ is a Morin singular point of type $A_1$, then $F$ is called a \emph{$\partial$-fold map}.
\end{defi}

Such maps appear in many previous studies on maps on manifolds with boundary.
We recall some of them here.

\begin{example}
A real-valued function $F$ on a manifold with boundary $N$ is called an $m$-function if all singular points of $F$ and $F|_{\partial N}$ are non-degenerate and $S(F)\cap\partial N=\emptyset$~\cite{Haj, JR}.
Such a map is a $\partial$-Morin map; in particular, it is a $\partial$-fold map~\cite{JR}.
Under the additional assumption that $F|_{S(F)\sqcup S(F|_{\partial N})}$ is injective, this definition agrees with that of a stable map to $\mathbb{R}$~\cite{Wra}.
\end{example}

\begin{example}
Singular points of stable maps from $3$-manifolds with boundary to $\mathbb{R}^2$ are classified in \cite{ NM, Shi}.
If the singular point set is disjoint from the boundary, or equivalently, if the stable map has no singular point of type $B_2$, then the map is a $\partial$-Morin map. 
For the definition of singular points of type $B_2$, see~\cite{Yam}.
Note that when the source manifold is orientable, every stable map is homotopic to a stable map with no singular points of type $B_2$~\cite{Yam}, that is, it is homotopic to a $\partial$-Morin map.
\end{example}

For $\partial$-Morin maps, the following proposition holds.
Since the proof is the same as that given by \cite{Fuk}, we omit it.

\begin{prop}
Let $F\colon N\to L$ be a $\partial$-Morin map.
Then, for every integer $k$ with $1\leq k\leq l$, the following hold: 
\begin{enumerate}[
  label=\normalfont(\arabic*),
  leftmargin=*,
  labelsep=.5em,
  itemsep=.3em,
  topsep=.45em,
  parsep=0pt, 
  leftmargin=2em
]
\item $A_k(F)$ and $\overline{A_k(F)}$ are $(l-k)$-dimensional submanifolds of $N$.
\item $A_k(F|_{\partial N})$ and $\overline{A_k(F|_{\partial N})}$ are $(l-k)$-dimensional submanifolds of $\partial N$.
\item $\overline{A_k(F)}=\bigcup_{i\geq k}A_i(F)$.
\item $\overline{A_k(F|_{\partial N})}=\bigcup_{i\geq k}A_i(F|_{\partial N})$.
\item The maps $F|_{A_k(F)}\colon A_k(F)\to L$ and $F|_{A_k(F|_{\partial N})}\colon A_k(F|_{\partial N})\to L$ are immersions.
\end{enumerate}
\end{prop}

\section{Gluing two smooth maps on manifolds with boundary}
In this section, we introduce a method for gluing two maps defined on manifolds with boundary, generalizing the method introduced in Levine~\cite{Lev2}.
Then, we show that gluing two $\partial$-Morin maps yields a Morin map on the closed manifold obtained by gluing the two manifolds with boundary. 
Furthermore, we describe the properties of the singular points of the resulting map in terms of those of the original maps.

Throughout this section, let $N_1$ and $N_2$ be compact $n$-dimensional manifolds with boundary satisfying $\partial N_1=\partial N_2$, let $L$ be an $l$-dimensional manifold, where $n>l$. 
We denote the boundary $\partial N_1=\partial N_2$ by $M$. 
Let $F_1\colon N_1\to L$ and $F_2\colon N_2\to L$ be maps, that are not necessarily $\partial$-Morin maps, that agree on collar neighborhoods of the boundaries of $N_1$ and $N_2$, that is, for collar neighborhoods $\phi_1\colon M\times[0,1)\to N_1$ and $\phi_2\colon M\times[0,1)\to N_2$ of the boundaries of $N_1$ and $N_2$, we assume that $F_1\circ\phi_1=F_2\circ\phi_2$. 
Moreover, we assume that $F_1\circ\phi_1=F_2\circ\phi_2$ is a submersion.

\subsection{Gluing construction}
In this subsection, we give a construction of a map on a closed manifold by gluing two maps on manifolds with boundary. 
This construction is a generalization of the method used by Levine~\cite{Lev2}.

We first recall the construction of the closed manifold obtained by gluing two manifolds $N_1$ and $N_2$ along their boundaries by the identity map. 
For more details, see \cite[Theorem 9.29]{Lee}.
We denote the topological space obtained from $N_1$ and $N_2$ identifying $x\in\partial N_1$ with $x\in\partial N_2$ by $N_1\cup_{\id}N_2$. 
Let $\pi\colon N_1\sqcup N_2\to N_1\cup_{\id} N_2$ be the quotient map. 
This topological space can be endowed with a smooth structure as follows: 
Denote $V_i={\rm Im\;}\phi_i$. 
Then, we define $\Phi\colon M\times(-1,1)\to \pi(V_1\sqcup V_2)$ as
$$
\Phi(x,t)=
\begin{cases}
\pi(\phi_1(x,-t)), & t\leq 0,\\
\pi(\phi_2(x,t)), & t\geq 0,
\end{cases}
$$
and this is a homeomorphism. 
Moreover, the restrictions $\pi|_{\Int N_1}$ and $\pi|_{\Int N_2}$ are injective. 
Thus, by these maps, the topological space $N_1\cup_{\id} N_2$ becomes an $n$-dimensional topological manifold. 
Furthermore, the smooth structure on $N_1\cup_{\id} N_2$ is defined so that these maps are diffeomorphisms.

We now construct a smooth map $F_1\cup F_2$ by gluing $F_1$ and $F_2$ along the boundary. 
Before doing this, we choose an even smooth function $g\colon(-1,1)\to[0,1)$ satisfying the following conditions. 
For some sufficiently small $\epsilon$ with $0<\epsilon<1/2$, we have
$$
g(t)=
\begin{cases}
t^2, & 0\leq |t|<\epsilon,\\
|t|, & 1-\epsilon<|t|< 1,
\end{cases}
$$
and
$$
t g'(t)>0
$$
for every $t\neq 0$.
Using this function, we define $F_1\cup F_2$ as follows. 
Away from $\pi(M)$, we set 
$$
(F_1\cup F_2)(\pi(x))
=
\begin{cases}
F_1(x), &  x\in N_1\setminus \phi_1(M\times[0,1-\epsilon)), \\
F_2(x), &  x\in N_2\setminus \phi_2(M\times[0,1-\epsilon)).
\end{cases}
$$
We next suppose that $\pi(x)\in\Phi(M\times[-1+\epsilon/2,1-\epsilon/2])$. 
Then, there exists $y\in M$ and $t\in[-1+\epsilon/2,1-\epsilon/2]$ such that $\pi(x)=\Phi(y,t)$. 
In this case, we define
$$
(F_1\cup F_2)(\pi(x))
=
\begin{cases}
F_1\circ\phi_1(y,g(-t)), &  t\leq 0,\\ 
F_2\circ\phi_2(y,g(t)), &  t>0.
\end{cases}
$$
Note that $F_1\circ\phi_1=F_2\circ\phi_2$ near $\Phi(M\times\{0\})$. 
Hence, the above definition gives a smooth map $F_1\cup F_2\colon N_1\cup_{\id} N_2\to L$.
When $F_1=F_2(=F)$, we denote this map by $D(F)$, and call this map the \emph{double} of $F$.

The singular point set of the resulting map $F_1\cup F_2$ is described in terms of those of $F_1$ and $F_2$, and the restriction to the boundary. 
This follows immediately from the computation of the Jacobian matrix of $F_1\cup F_2$.

\begin{lem}\label{Lem3.1}
We have
$$
S(F_1\cup F_2)=\pi(S(F_1))\sqcup \pi(S(F_2))\sqcup\pi(S(F_1|_{M})). 
$$
\end{lem}

\begin{rem}
In this paper, when $N_1$ and $N_2$ are oriented, we consider $N_1\cup_{\id} \overline{N_2}$.
Here, $\overline{N_2}$ denotes $N_2$ with the opposite orientation. 
Then, the resulting manifold carries the orientation compatible with those of $N_1$ and $N_2$. 
\end{rem}

\subsection{Gluing two $\partial$-Morin maps}
In this subsection, we study the singular points of the map on the closed manifold obtained by gluing two $\partial$-Morin maps, as constructed in Subsection~3.1.
We first prove the following lemma. 
This lemma generalizes the results in \cite{Shi}.

\begin{lem}\label{Prop3.3}
Let $N$ be a compact $n$-dimensional manifold with boundary, let $L$ be an $l$-dimensional manifold, and let $F\colon N\to L$ be a $\partial$-Morin map, where $n>l$. 
For any point $p\in A_k(F|_{\partial N})$, there exist local coordinates $(x_1,\dots,x_{n-1}, x_{n})$ of $N$ around $p$ and $(y_1,\dots,y_{l})$ of $L$ around $F(p)$ such that 
$$
y_i\circ F(x_1,\dots,x_{n})=x_i, \quad 1\leq i\leq l-1, 
$$
and
$$
y_l\circ F(x_1,\dots,x_{n})=x_{l}^{k+1}+\sum_{r=1}^{k-1}x_r x_l^{k-r}-\sum_{r=l+1}^{l+\lambda}x_r^2+\sum_{r=l+\lambda+1}^{n-1}x_r^2+\sigma x_n, 
$$
where $x_n>0$ and $x_n=0$ correspond to $\Int N$ and $\partial N$, respectively.
Here, $\sigma\in\{1, -1\}$ and $0\leq\lambda\leq n-l-1$. 
\end{lem}

\begin{proof}
Let $p\in A_k(F|_{\partial N})$. 
By the local normal form for a Morin singular point, there exist a collar neighborhood $\phi\colon\partial N\times[0,1)\longrightarrow N$, and local coordinates $(x_1,\dots,x_{n-1})$ of $\partial N$ around $p$ and $(y_1,\dots,y_l)$ of $L$ around $F(p)$ such that
$$
y_i\circ F\circ\phi(x_1,\dots,x_{n-1},0)=x_i,
\qquad
1\leq i\leq l-1,
$$
and
$$
y_l\circ F\circ\phi(x_1,\dots,x_{n-1},0)
=
x_l^{k+1}
+\sum_{r=1}^{k-1}x_r x_l^{k-r}
-\sum_{r=l+1}^{l+\lambda}x_r^2
+\sum_{r=l+\lambda+1}^{n-1}x_r^2.
$$
Here, $s=0$ and $s>0$ correspond to $\partial N$ and $\Int N$, respectively.
Denote the expression on the right-hand side of the second equation by $H(x_1,\dots,x_{n-1})$.

Since $p\in \partial N$, the point $p$ is a regular point of $F$. 
Together with the above normal form for $F|_{\partial N}$, this implies that
$$
\gamma=
\frac{\partial(y_l\circ F\circ\phi)}{\partial s}(0)
\neq 0.
$$
Set $\delta\in\{-1,1\}$ the sign of $\gamma$.

Define a map $\theta$ on a neighborhood of the origin in $\mathbb{R}^n$ by
$$
\theta(x_1,\dots,x_{n-1},s)
=
(X_1,\dots,X_{n-1},S),
$$
where
$$
X_i=
\begin{cases}
y_i\circ F\circ\phi(x_1,\dots,x_{n-1},s),
& 1\leq i\leq l-1,\\
x_i,
& l\leq i\leq n-1,
\end{cases}
$$
and
$$
S
=
\delta\Big(
y_l\circ F\circ\phi(x_1,\dots,x_{n-1},s)
-
H(X_1,\dots,X_{n-1})
\Big).
$$
When $s=0$, we have $X_i=x_i$ for $1\leq i\leq n-1$, and $S(x_1,\dots,x_{n-1},0)=0$.
Moreover, since $H$ has no linear terms, $dH_0=0$. 
Consequently, the Jacobian matrix of $\theta$ at $0$ is 
$$
J\theta_0
=
\begin{pmatrix}
I_{n-1} & *\\
0 & |\gamma|
\end{pmatrix}.
$$
Thus, $J\theta_0$ is invertible. 
By the inverse function theorem, after restricting to a sufficiently small neighborhood of the origin, $\theta$ is a diffeomorphism onto its image.
Furthermore, since $\frac{\partial S}{\partial s}(0)=|\gamma|>0$, after restricting the neighborhood of the origin if necessary, we may assume that 
$$
\frac{\partial S}{\partial s}>0
$$ 
throughout the neighborhood. 
Since $S(x_1,\dots,x_{n-1},0)=0$, it follows that $S>0$ if and only if $s>0$, and $S=0$ if and only if $s=0$.
Hence, $\theta$ preserves the boundary and the interior.

In the new coordinates $(X_1,\dots,X_{n-1},S)$, the map $F$ is represented as
$$
y\circ F\circ\phi\circ\theta^{-1}
(X_1,\dots,X_{n-1},S)
=
\big(
X_1,\dots,X_{l-1},
H(X_1,\dots,X_{n-1})+\delta S
\big),
$$
which is the desired local normal form.
\end{proof}

\begin{rem}
Suppose that $k=1$ in the preceding lemma. 
Thus, let $p$ be a fold point. 
In suitable local coordinates around $p$ and $F(p)$, $F$ has the form
$$
(x_1,\ldots,x_n)
\longmapsto
\Big(
x_1,\ldots,x_{l-1},
-\sum_{r=l}^{l+\lambda-1}x_r^2
+\sum_{r=l+\lambda}^{n-1}x_r^2
+\sigma x_n
\Big),
$$
where $\sigma\in\{-1,1\}$. 
If the coefficients of all the quadratic terms $x_i^2$ and the linear term $x_n$ have the same sign, then $p$ is called a \emph{boundary definite fold point}; otherwise, it is called a \emph{boundary indefinite fold point}~\cite{Iwa3, SY2}.
\end{rem}

Using the preceding lemma, we can determine the local normal forms of $F_1\cup F_2$ around its singular points. 
The following proposition is the basic result underlying the applications in the subsequent sections.
The proof is based on that of \cite{Lev2}.

\begin{prop}\label{Lem3.5}
Let $F_1\colon N_1\to L$ and $F_2\colon N_2\to L$ be $\partial$-Morin maps. 
Then, the following holds for the singular points of $F_1\cup F_2$: 
\begin{enumerate}[
  label=\normalfont(\arabic*),
  leftmargin=*,
  labelsep=.5em,
  itemsep=.3em,
  topsep=.45em,
  parsep=0pt, 
  leftmargin=2em
]
\item If $p\in A_k(F_j)$, then $\pi(p)\in A_k(F_1\cup F_2)$.
More precisely, if $F_j$ has the local normal form around $p$ given in \Cref{Def2.5}, then $F_1\cup F_2$ has the same local normal form around $\pi(p)$. 
\item If $p\in A_k(F_j|_{\partial N_j})$, then $\pi(p)\in A_k(F_1\cup F_2)$. 
More precisely, if $F_j|_{\partial N_j}$ has the local normal form around $p$ in \Cref{Prop3.3}, then there exist local coordinates $(x_1,\dots, x_{n})$ of $N_1\cup N_2$ around $\pi(p)$ and $(y_1,\dots,y_{l})$ of $L$ around $(F_1\cup F_2)(\pi(p))$ such that
$$
y_i\circ (F_1\cup F_2)=x_i, \quad 1\leq i\leq l-1, 
$$
and
$$
y_l\circ (F_1\cup F_2)=x_{l}^{k+1}+\sum_{r=1}^{k-1}x_r x_l^{k-r}-\sum_{r=l+1}^{l+\lambda}x_r^2+\sum_{r=l+\lambda+1}^{n-1}x_r^2+\sigma x_n^2, 
$$
where $\lambda$ and $\sigma$ are as in \Cref{Prop3.3}. 
\end{enumerate}
In particular, $F_1\cup F_2\colon N_1\cup_{\id}N_2\to L$ is a Morin map. 
\end{prop}

\begin{proof}
The statement (1) follows immediately from the definition of the map $F_1\cup F_2$.

We next prove the statement (2). 
Fix $j\in\{1,2\}$, and let $p\in A_k(F_j|_{\partial N_j})$. 
Then, there exist a collar neighborhood $\phi\colon\partial N_j\times[0,1)\to N_j$, and local coordinates $(x_1,\dots,x_{n-1})$ of $\partial N$ around $p$ and $(y_1,\dots, y_l)$ of $L$ around $F_j(p)$ such that 
$$
y_i\circ F_j\circ \phi(x_1,\dots, x_{n-1},0)=x_i, 
\qquad 
1\leq i\leq l-1, 
$$
and 
$$
y_l\circ F_j\circ \phi(x_1,\dots, x_{n-1},0)=x_{l}^{k+1}+\sum_{r=1}^{k-1}x_r x_l^{k-r}-\sum_{r=l+1}^{l+\lambda}x_r^2+\sum_{r=l+\lambda+1}^{n-1}x_r^2.
$$
In what follows, we denote the right-hand side of the second equation by $H(x_1,\dots, x_{n-1})$.

By Hadamard’s lemma, there exist functions $B_i$ for $1\leq i\leq l$ defined on a sufficiently small neighborhood of the origin such that
$$
y_i\circ F_j\circ \phi(x,s)=
\begin{cases}
x_i+sB_i(x,s), 
& 1\leq i\leq l-1, \\
H(x)+sB_l(x,s), 
& i=l. 
\end{cases}
$$
For sufficiently small $|s|$, we have
$$
(F_1\cup F_2)\circ\Phi(x,s)=F_j\circ\phi(x,s^2), 
$$
where $\Phi$ is induced one by $\phi$ as Subsection~3.1. 
Consequently, with respect to the above local coordinates, we obtain
$$
y_i\circ (F_1\cup F_2)\circ\Phi(x,s)=
\begin{cases}
x_i+s^2 B_i(x,s^2), 
& 1\leq i\leq l-1, \\
H(x)+s^2 B_l(x,s^2), 
& i=l. 
\end{cases}
$$

Define a map $\theta$ on a neighborhood of the origin of $\mathbb{R}^n$ by
$$
\theta(x_1,\dots,x_{n-1},s)=(X_1,\dots,X_{n-1},S), 
$$
where
$$
X_i=
\begin{cases}
x_i+s^2 B_i(x,s^2), 
& 1\leq i\leq l-1, \\
x_i, 
& l\leq i\leq n-1, 
\end{cases}
$$
and $S=s$. 
Since the Jacobian matrix of $\theta$ at the origin is the identity matrix, the inverse function theorem implies that $\theta$ is a local diffeomorphism around the origin. 
Using $\theta$, we have
$$
y_i\circ(F_1\cup F_2)\circ\Phi\circ\theta^{-1}(X,S)=
\begin{cases}
X_i, 
& 1\leq i\leq l-1, \\
H(X)+S^2 C(X,S), 
& i=l, 
\end{cases}
$$
where 
$$
C(X,S)=B_{l}(x,S^2)-\sum_{r=1}^{k-1}B_r(x,S^2)X_l^{k-r}.
$$

Since $p$ is a regular point of $F_j$, the Jacobian matrix of $F_j$ at $p$ has the maximal rank. 
This implies that $B_l(0)\neq 0$, and we set the sign of $B_l(0)$ as $\delta\in\{-1,1\}$. 
Since $C(0)\neq 0$, continuity implies that $\delta C>0$ around the origin. 
Then, we define $T=S\sqrt{\delta C(X,S)}$ and $\eta(X,S)=(X,T)$. 
The map $\eta$ is a local diffeomorphism around the origin. 
Indeed, since $\frac{\partial T}{\partial S}(0)=\sqrt{|B_l(0)|}\neq 0$, the Jacobian matrix of $\eta$ at the origin is invertible. 
Hence, $\eta$ is a local diffeomorphism in a neighborhood of the origin. 
Using $\eta$, we obtain
$$
y_i\circ(F_1\cup F_2)\circ\Phi\circ\theta^{-1}\circ\eta^{-1}(X,T)
=
\begin{cases}
X_i, 
& 1\leq i\leq l-1, \\
H(X)+\delta T^2, 
& i=l, 
\end{cases}
$$
which is the desired local normal form.

Finally, since $\delta C>0$ in a neighborhood of the origin, $s=0$ if and only if $T=0$, and $s>0$ if and only if $T>0$. 
Thus, the coordinate change preserves the boundary and the interior.
\end{proof}

\begin{rem}
The Morin map $F_1\cup F_2$ is not a stable map in general. 
If $(n,l)$ lies in the nice dimensions, then the stable maps are dense~\cite{Mat6}. 
Hence, $F_1\cup F_2$ can be approximated by a stable map from $N_1\cup_{\id}N_2$ to $L$. 
Since Morin singular points are stable as map germs, we may take such an approximation sufficiently close to $F_1\cup F_2$ so that its singular point set is diffeomorphic to $S(F_1\cup F_2)$. 
By abuse of notation, we denote this stable map again by $F_1\cup F_2$. 
\end{rem}

\section{Relative formulas for $\partial$-Morin maps}
In this section, we establish relative versions of Fukuda's theorem and Sakuma's theorem by using the gluing construction developed in the preceding section.

\subsection{Relative version of Fukuda’s theorem}
In this subsection, we establish a formula for the Euler characteristic in terms of the singular points of a $\partial$-Morin map and its restriction to the boundary. 
We first recall the results of Ruiz~\cite{Rui}, which include Fukuda’s theorem~\cite{Fuk} as a special case.

\begin{thm}[Ruiz]\label{Thm4.1}
Let $M$ be an $m$-dimensional closed manifold and let $f\colon M\to\mathbb{R}^l$ be a Morin map, where $m\geq l$. 
If $m-l$ is odd, then 
$$
\chi(M)=
\sum_{k:\vspace{1pt}\text{odd}}
\bigl(
\chi(\overline{A_k^+(f)})-\chi(\overline{A_k^-(f)})
\bigr).
$$
Here, 
\begin{align*}
A_k^+(f)=\{p\in A_k(f)\mid \nu(p)\equiv 0 \pmod 2\},\\
A_k^-(f)=\{p\in A_k(f)\mid \nu(p)\equiv 1 \pmod 2\}, 
\end{align*}
where $\nu(p)$ is the number of negative signs among the coefficients of $x_l^{k+1}$ and the quadratic terms in the local normal form in \Cref{Def2.5}. 
\end{thm}

\begin{rem}\label{Rem4.2}
When $m-l$ is odd, the sets $A_k^\pm(f)$ used in the above theorem are well-defined, independently of the choice of local coordinates.
Indeed, in the local normal form in \Cref{Def2.5}, we have $\nu(p)=\lambda$. 
On the other hand, if we consider the local normal form obtained by multiplying the last component by $-1$, then the integer is $\nu(p)=m-l-\lambda+1$. 
Since $m-l$ is odd, we have $\lambda\equiv m-l-\lambda+1 \pmod 2$.
Thus, the parity of $\nu(p)$ is independent of the choice of the local coordinates. 
\end{rem}

Then, we obtain our theorem, which is the relative version of Fukuda's theorem.

\begin{thm}\label{Thm4.3}
Let $F\colon N\to\mathbb{R}^l$ be a $\partial$-Morin map of a compact $n$-dimensional manifold $N$ with boundary, where $n>l$. 
If $n$ is even and $l$ is odd, then we have
$$
\chi(N)
=
\sum_{k:\vspace{1pt}\text{odd}}\Big(\chi(\overline{A_k^+(F)})-\chi(\overline{A_k^-(F)})\ \Big)
+
\frac{1}{2}\sum_{k:\vspace{1pt}\text{odd}}\Big(\chi(\overline{{}^{\partial\!}A_k^+(F)})-\chi(\overline{{}^{\partial\!}A_k^-(F)})\Big). 
$$
where $A^{\pm}_k(F)$ are defined as above.
Moreover, 
\begin{align*}
{}^{\partial\!} A_k^+(F)=\{p\in A_k(F|_{\partial N})\mid {}^{\partial\!}\nu(p)\equiv 0 \pmod 2\},\\
{}^{\partial\!} A_k^-(F)=\{p\in A_k(F|_{\partial N})\mid {}^{\partial\!}\nu(p)\equiv 1 \pmod 2\}, 
\end{align*}
where ${}^{\partial\!}\nu(p)$ is the number of negative signs among the coefficients of $x_l^{k+1}$, the quadratic terms in the local normal form, and the linear term in \Cref{Prop3.3}. 
\end{thm}

\begin{proof}
By the gluing construction in Section~3, we have a Morin map $D(F)\colon D(N)\to\mathbb{R}^l$. 
Applying \Cref{Thm4.1} together with \Cref{Prop3.3} and \Cref{Lem3.5}, we obtain
$$
\chi(D(N))=
2\sum_{k:\vspace{1pt}\text{odd}}\bigl(\chi(\overline{A_k^+(F)})-\chi(\overline{A_k^-(F)})\bigr)
+
\sum_{k:\vspace{1pt}\text{odd}}\bigl(\chi(\overline{{}^{\partial\!} A_k^+(F)})-\chi(\overline{{}^{\partial\!} A_k^-(F)})\bigr).
$$
Since $n$ is even, we have $\chi(\partial N)=0$. 
Therefore, $\chi(D(N))=2\chi(N)$. 
Dividing the above equation by $2$, we obtain the desired formula. 
\end{proof}

\begin{rem}
As in \Cref{Rem4.2}, one verifies that the parity of ${}^{\partial}\!\nu(p)$ is also independent of the choice of local coordinates when $n-l$ is odd.
\end{rem}

The preceding discussion deals only with maps whose target manifold is a Euclidean space. 
However, analogous formulas can also be obtained for maps into general manifolds, provided that appropriate assumptions are imposed on the types of singular points or on the dimensions of the manifolds. 
Saeki~\cite{Sae3} established Fukuda’s formula for fold maps into general manifolds. 
Dutertre--Fukui~\cite{DF} obtained a corresponding formula for stable Morin maps into general manifolds.
We refer the reader to these papers for further details.
Using these results, we obtain the following two theorems by arguments similar to the proof of \Cref{Thm4.3}.

\begin{thm}\label{Thm4.5}
Let $N$ be a compact $n$-dimensional manifold with boundary, let $L$ be a connected $l$-dimensional manifold, and let $F\colon N\to L$ be a $\partial$-fold map. 
If $n$ is even and $l$ is odd, then 
$$
\chi(N)=\chi(A_1^+(F))-\chi(A_1^-(F))
+
\frac{1}{2}
\bigl(\chi({}^{\partial\!} A_1^+(F))-\chi({}^{\partial\!} A_1^-(F))\bigr),
$$
where the notation is the same as in \Cref{Thm4.3}. 
\end{thm}

\begin{thm}\label{Thm4.6}
Let $N$ be a compact $n$-dimensional manifold with boundary, let $L$ be a connected $l$-dimensional manifold, and let $F\colon N\to L$ be a $\partial$-Morin map. 
If $n$ is even, $l$ is odd, and $(n,l)$ belongs to the nice dimensions, then 
$$
\chi(N)=
\sum_{k:\vspace{1pt}\text{odd}}
\bigl(
\chi(\overline{A_k^+(F)})-\chi(\overline{A_k^-(F)})
\bigr)
+
\frac{1}{2}
\sum_{k:\vspace{1pt}\text{odd}}
\bigl(
\chi(\overline{{}^{\partial\!}A_k^+(F)})-\chi(\overline{{}^{\partial\!}A_k^-(F)})
\bigr), 
$$
where the notation is the same as in \Cref{Thm4.3}. 
\end{thm}

The preceding theorems are given under the assumption that $n$ is even and $l$ is odd. 
When $n$ and $l$ are odd, we obtain the following formula.
It follows directly from the equation $\chi(N)=\frac{1}{2}\chi(\partial N)$ and \Cref{Thm4.1}.

\begin{prop}
Let $N$ be a compact $n$-dimensional manifold with boundary and let $F\colon N\to\mathbb{R}^l$ be a $\partial$-Morin map. 
If $n$ and $l$ are odd, then
$$
\chi(N)
=
\frac{1}{2}\sum_{k:\vspace{1pt}\text{odd}}\bigl(\chi(\overline{A_k^+(F|_{\partial N})})-\chi(\overline{A_k^-(F|_{\partial N})}\bigr), 
$$
where the notation is the same as in \Cref{Thm4.1}. 
\end{prop}

Similarly, using the results of Saeki~\cite{Sae3} and Dutertre-Fukui~\cite{DF}, one can obtain analogous formulas for maps whose targets are general manifolds.

\subsection{Relative version of Sakuma's theorem}
In this subsection, using the relative version of Fukuda's theorem in Subsection~4.1, we derive a relative version of Sakuma's theorem~\cite{Sak}. 
We first give our theorem, which is the relative version of Sakuma's theorem.

\begin{thm}\label{Thm4.8}
Let $N$ be a compact oriented $4$-manifold whose boundary is an integral homology sphere and let $L$ be a connected parallelizable $3$-manifold. 
If there exists a $\partial$-fold map $F\colon N\to L$, then 
$$
\sigma(N)+S(F)\cdot S(F)\equiv 2+\chi({}^{\partial\!}A_1^+(F))-\chi({}^{\partial\!}A_1^-(F)) \pmod 4, 
$$
where $\sigma(N)$ is the signature of $N$, $S(F)\cdot S(F)$ is the self-intersection number of $S(F)$ in $N$, and ${}^{\partial\!}A_1^{\pm}(F)$ are the same as those in \Cref{Thm4.3}.
\end{thm}

\begin{proof}
We first prepare three lemmas. 
For a closed characteristic surface embedded in $N$, Klug~\cite[Theorem~6]{Klu} and Kirby-Taylor~\cite[Lemma~3.6]{KT} give the following relation between its self-intersection number and the signature of $N$.

\begin{lem}\label{Lem4.9}
Let $N$ be a compact oriented $4$-manifold whose boundary is an integral homology sphere, and let $S\subset \Int N$ be a closed characteristic surface. 
Then, 
$$
2\chi(S)
\equiv 
\sigma(N)-S\cdot S \pmod 4. 
$$
\end{lem}

To apply the preceding lemma to $S(F)\subset\Int N$, we establish the following lemma.

\begin{lem}\label{Lem4.10}
Let $N$ be a compact oriented $4$-manifold with boundary, let $L$ be a parallelizable $3$-manifold, and let $F\colon N\to L$ be a $\partial$-fold map. 
Then, $S(F)\subset\Int N$ is a characteristic surface. 
\end{lem}

\begin{proof}
Throughout this proof, all cohomology and homology are taken with coefficients in $\mathbb{Z}_2$.
Applying \cite[Theorems~5.26 and~9.2]{Tan} to $F$, we obtain 
$$
i^*\mathrm{PD}(j_*[S(F)])=w_2(N), 
$$ 
where, $i\colon (N,\emptyset)\to (N,\partial N)$ and $j\colon S(F)\to N$ are the natural inclusions, and $i^*$ and $j_*$ are the induced homomorphisms in cohomology and homology, respectively.
Moreover, $\mathrm{PD}\colon H_2(N)\to H^2(N,\partial N)$ is the Poincar\'e duality isomorphism.
Let $\iota$ denote the natural inclusion $i\circ j\colon (S(F),\emptyset)\to(N,\partial N)$.
It follows that ${\rm PD}(\iota_*[S(F)])=w_2(N)$ by the naturality of Poincar\'e duality.
In the notation of Klug, we have $\iota_*[S(F)]=[S(F),\emptyset]$, and hence, ${\rm PD}( [S(F),\emptyset])=w_2(N)$.
Therefore, $S(F)\subset\Int N$ is a characteristic surface. 
\end{proof}

Moreover, for a compact oriented $4$-manifold whose boundary is an integral homology sphere, its signature and Euler characteristic satisfy the following congruence.

\begin{lem}\label{Lem4.11}
Let $N$ be a compact oriented $4$-manifold whose boundary is an integral homology sphere. 
Then, 
$$
\chi(N)\equiv\sigma(N)+1 \pmod 2. 
$$ 
\end{lem}

\begin{proof}
Since $\partial N$ is a homology sphere, we have $b_0(N)-b_4(N)=1$ and $b_1(N)=b_3(N)$. 
Hence, $\chi(N)\equiv 1+b_2(N)\pmod2$.
On the other hand, we have $\sigma(N)\equiv b_2(N)\pmod 2$ by the definition of the signature.
Combining these congruences, we obtain the desired one.
\end{proof}

We prove the theorem using the preceding lemmas. 
Reducing the formula in \Cref{Thm4.5} modulo $4$, we obtain 
$$
2\chi(N)\equiv 2\chi(S(F))
+
\chi({}^{\partial\!}A_1^+(F))-\chi({}^{\partial\!}A_1^-(F)) \pmod4. 
$$
On the other hand, by \Cref{Lem4.9} and \Cref{Lem4.10}, we have
$$
2\chi(S(F))\equiv \sigma(N)-S(F)\cdot S(F) \pmod 4. 
$$
Combining these two congruences yields
$$
2\chi(N)\equiv\sigma(N)-S(F)\cdot S(F)+ 
\chi({}^{\partial\!}A_1^+(F))-\chi({}^{\partial\!}A_1^-(F)) \pmod4.
$$
Moreover, substituting the congruence in \Cref{Lem4.11} into the preceding one, we obtain
$$
\sigma(N)+S(F)\cdot S(F)\equiv
2+\chi({}^{\partial\!}A_1^+(F))-\chi({}^{\partial\!}A_1^-(F)) \pmod 4.
$$
This is the desired congruence. 
\end{proof}

\begin{rem}
The congruence obtained by Sakuma~\cite{Sak} was refined to an equality by Ohmoto--Saeki--Sakuma~\cite{OSS}. 
It remains unknown whether the congruence established here admits an analogous refinement to an equality.
\end{rem}

We next establish a relation with the Hirzebruch defect. 
To this end, we define a stable framing of a $3$-manifold induced by a submersion. 
We refer to \cite{Ati, KM} for the definition of a framing and a stable framing.
Let $N$ be a compact oriented $4$-manifold with boundary, and let $F\colon N\to\mathbb{R}^3$ be a map that has no singular points around the boundary.
More precisely, let $V$ be a collar neighborhood of $\partial N$ such that the restriction $F|_{V}\colon V\to\mathbb{R}^3$ is a submersion.
Choose a nowhere vanishing section $\xi_0$ of the line bundle $\ker dF|_V$. 
Let $dx_1$, $dx_2$, and $dx_3$ denote the standard $1$-forms on $\mathbb{R}^3$, and consider their pullbacks $(F|_V)^*(dx_1)$, $(F|_V)^*(dx_2)$, and $(F|_V)^*(dx_3)$ on $V$. 
After fixing a Riemannian metric on $V$, let $\xi_1$, $\xi_2$, and $\xi_3$ be the vector fields on $V$ dual to these $1$-forms.
Then, $\phi=(\xi_0, \xi_1, \xi_2, \xi_3)$ defines a framing of $TV$.
Restricting this framing to $\partial N$, we obtain a stable framing of $\partial N$, which we denote by $\phi_{\partial N}$.

\begin{cor}
Let $N$ be a compact oriented $4$-manifold whose boundary is an integral homology sphere. 
Suppose that there exists a stable $\partial$-fold map $F\colon N\to\mathbb{R}^3$ whose singular point set is oriented.
Then,  
$$
h(\phi_{\partial N})\equiv 
2+\chi({}^{\partial\!}A_1^+(F))-\chi({}^{\partial\!}A_1^-(F)) \pmod 4, 
$$
where $h(\phi_{\partial N})$ denotes the Hirzebruch defect of the stable framing $\phi_{\partial N}$ induced by $F|_{V}$.
\end{cor}

\begin{proof}
By \cite[Lemma~4.1]{KS}, we have $h(\phi_{\partial N})=S(F)\cdot S(F)-3\sigma(N)$.
Therefore, the desired congruence follows from \Cref{Thm4.8}.
\end{proof}

\section{Singular fibers and the topology of $3$-manifolds with boundary}
In this section, we establish a relation between the simplicial volume of a compact $3$-manifold with boundary and the singular fibers of a stable $\partial$-fold map into $\mathbb{R}^2$. 
As an application, we show that under suitable conditions on the map, any compact oriented irreducible $3$-manifold with a toroidal boundary admitting such a map is a graph manifold.

\subsection{Classification of singular fibers and their behavior under doubling}
In this subsection, we first review the classifications of fibers of a stable map from closed $3$-manifolds to the plane, due to Kushner--Levine--Porto~\cite{KLP}, and that of maps from compact $3$-manifolds with boundary into the plane, due to Saeki--Yamamoto~\cite{SY2}. 
Then, we determine how the fibers change under the double.

We begin by recalling the classifications by Kushner--Levine--Porto and Saeki--Yamamoto. 
For this purpose, we introduce the notion of a fiber of a smooth map, and an equivalence relation between fibers. 
The following definitions are taken from \cite{Sae6, SY2}.

\begin{defi}
Let $M$ be an $m$-dimensional manifold with possibly non-empty boundary, and let $L$ be an $l$-dimensional manifold without boundary, where $m\geq l$. 
For a map $f\colon M\to L$ and a point $q\in L$, the map germ of $f$ along $f^{-1}(q)$, 
$$
f\colon (M,f^{-1}(q))\to(L,q)
$$
is called the \emph{fiber} of $f$ over $q$. 
\end{defi}

We define an equivalence relation between fibers as well.

\begin{defi}
For $j=0,1$, let $M_j$ be an $m$-dimensional manifold with possibly non-empty boundary, let $L_j$ be an $l$-dimensional manifold without boundary, and let $f_j\colon M_j\to L_j$ be maps, where $m\geq l$. 
Given points $q_j\in L_j$, the fibers of $f_0$ over $q_0$ and $f_1$ over $q_1$ are said to be \emph{equivalent} if there exist neighborhoods $U_j$ of $q_j$ in $L_j$ and diffeomorphisms $\theta\colon f_0^{-1}(U_0)\to f_1^{-1}(U_1)$ and $\eta\colon U_0\to U_1$ with $\eta(q_0)=q_1$ that make the following diagram commutative:
$$
\xymatrix{
(f^{-1}(U_0),f^{-1}(q_0)) \ar[r]^\theta \ar[d]_{f_0} & (f^{-1}(U_1),f^{-1}(q_1)) \ar[d]^{f_1} \\
(U_0,q_0) \ar[r]_\eta          & (U_1,q_1)
}
$$
\end{defi}

\begin{rem}
In \cite{Sae6, SY1, SY2}, the equivalence relation defined above is called \emph{$C^{\infty}$-equivalence}.
Replacing all the diffeomorphisms in the definition by homeomorphisms yields the notion of \emph{$C^{0}$-equivalence}. 
Since only $C^{\infty}$-equivalence is used in this paper, we refer to it simply as equivalence.   
\end{rem}

Up to the equivalence defined above, the fibers of stable maps from closed orientable $3$-manifolds to the plane are classified in \cite{KLP}, and an explicit description is given in \cite{Sae6}.
For the reader’s convenience, we recall the result from \cite{Sae6} below.

\begin{prop}[Kushner--Levine--Porto]\label{Prop5.4}
Let $M$ be a closed orientable $3$-manifold and let $f\colon M\to\mathbb{R}^2$ be a stable map. 
Then, every fiber of $f$ is equivalent to the disjoint union of one of the following fibers and finitely many copies of a fiber $0^0$ of a trivial circle bundle:
\begin{enumerate}[
  label=\normalfont(\arabic*),
  leftmargin=*,
  labelsep=.5em,
  itemsep=.3em,
  topsep=.45em,
  parsep=0pt, 
  leftmargin=2em
]
\item one of the fibers shown in~\Cref{Fig2};
\item a disconnected fiber ${\rm II}^{\mu,\nu}$, defined as the disjoint union of the fibers ${\rm I}^{\mu}$ and ${\rm I}^{\nu}$ shown in~\Cref{Fig2}, where $\mu, \nu\in\{0,1\}$.
\end{enumerate}
\end{prop}

\begin{figure}
    \centering
    \includegraphics[width=0.5\linewidth]{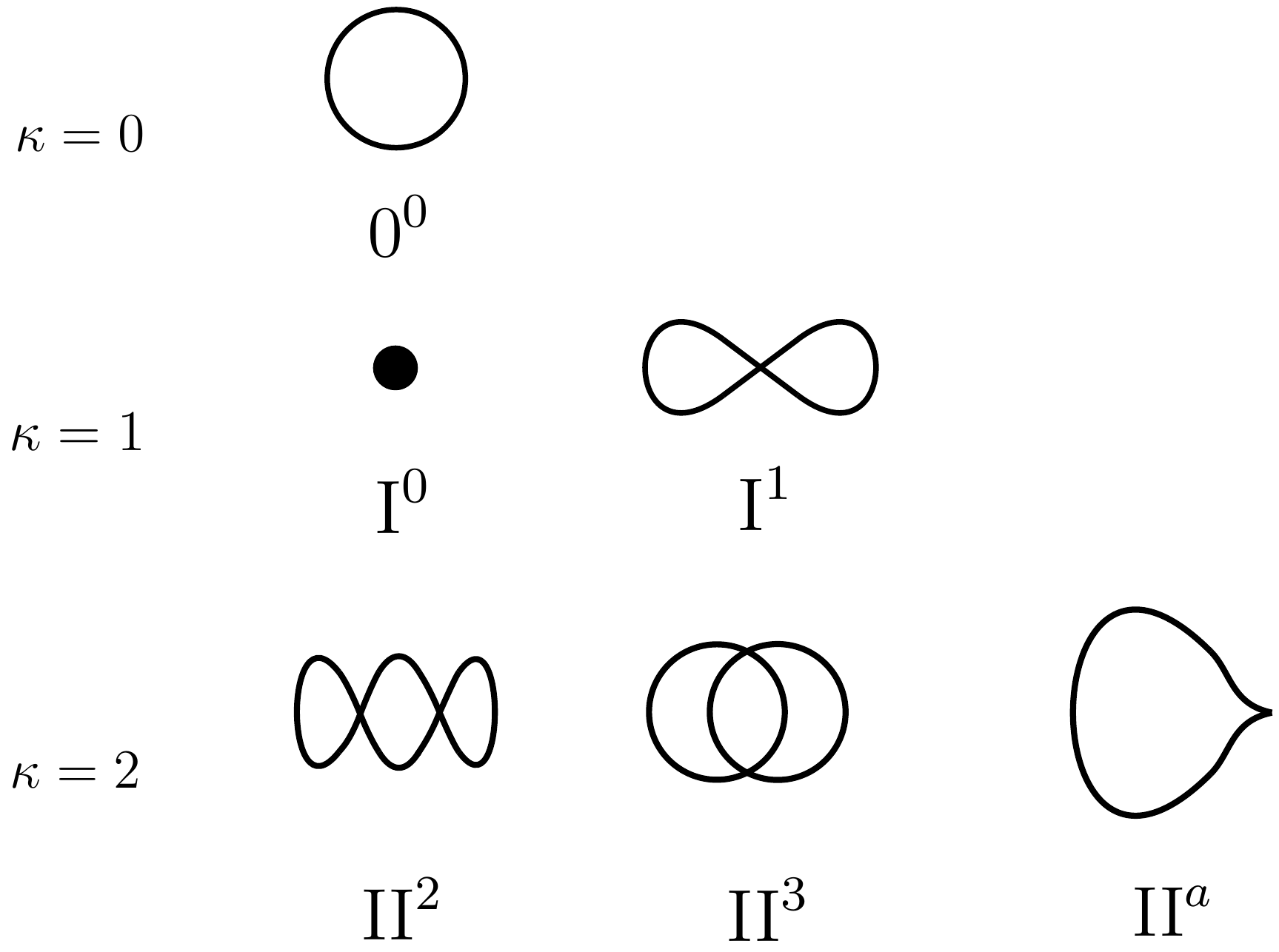}
    \caption{Fibers of stable maps from closed orientable $3$-manifolds to $\mathbb{R}^2$.}
    \label{Fig2}
\end{figure}

In \Cref{Fig2}, dots, crossings, and sharp points represent definite fold points, indefinite fold points, and cusp points, respectively. 
Moreover, $\kappa$ denotes the codimension in $L$ of the points corresponding to the fiber.
See \cite{Sae6} for further details.

Moreover, the fibers of stable maps from compact orientable $3$-manifolds with boundary to the plane are classified in \cite{SY2}.
For the reader’s convenience, we recall the classification given in \cite[Proposition~2.4, Proposition~2.6]{SY2}.

\begin{prop}[Saeki--Yamamoto]\label{Prop5.5}
Let $N$ be a compact orientable $3$-manifold with boundary and let $F\colon N\to\mathbb{R}^2$ be a stable map. 
Then, every fiber of $F$ is equivalent to the disjoint union of one of the following fibers and finitely many copies of regular fibers ${\rm b0}^0$ of a trivial circle bundle and ${\rm b0}^1$ of a trivial $I$-bundle, where $I=[-1,1]$:
\begin{enumerate}[
  label=\normalfont(\arabic*),
  leftmargin=*,
  labelsep=.5em,
  itemsep=.3em,
  topsep=.45em,
  parsep=0pt, 
  leftmargin=2em
]
\item one of the fibers shown in~\Cref{Fig3}; 
\item a disconnected fiber ${\rm bII}^{\mu,\nu}$, defined as the disjoint union of the fibers ${\rm bI}^{\mu}$ and ${\rm bI}^{\nu}$ shown in~\Cref{Fig3}, where $\mu, \nu\in\{2,3,4,5,6,7,8\}$. 
\end{enumerate}
\end{prop}

\begin{figure}
    \centering
    \includegraphics[width=0.69\linewidth]{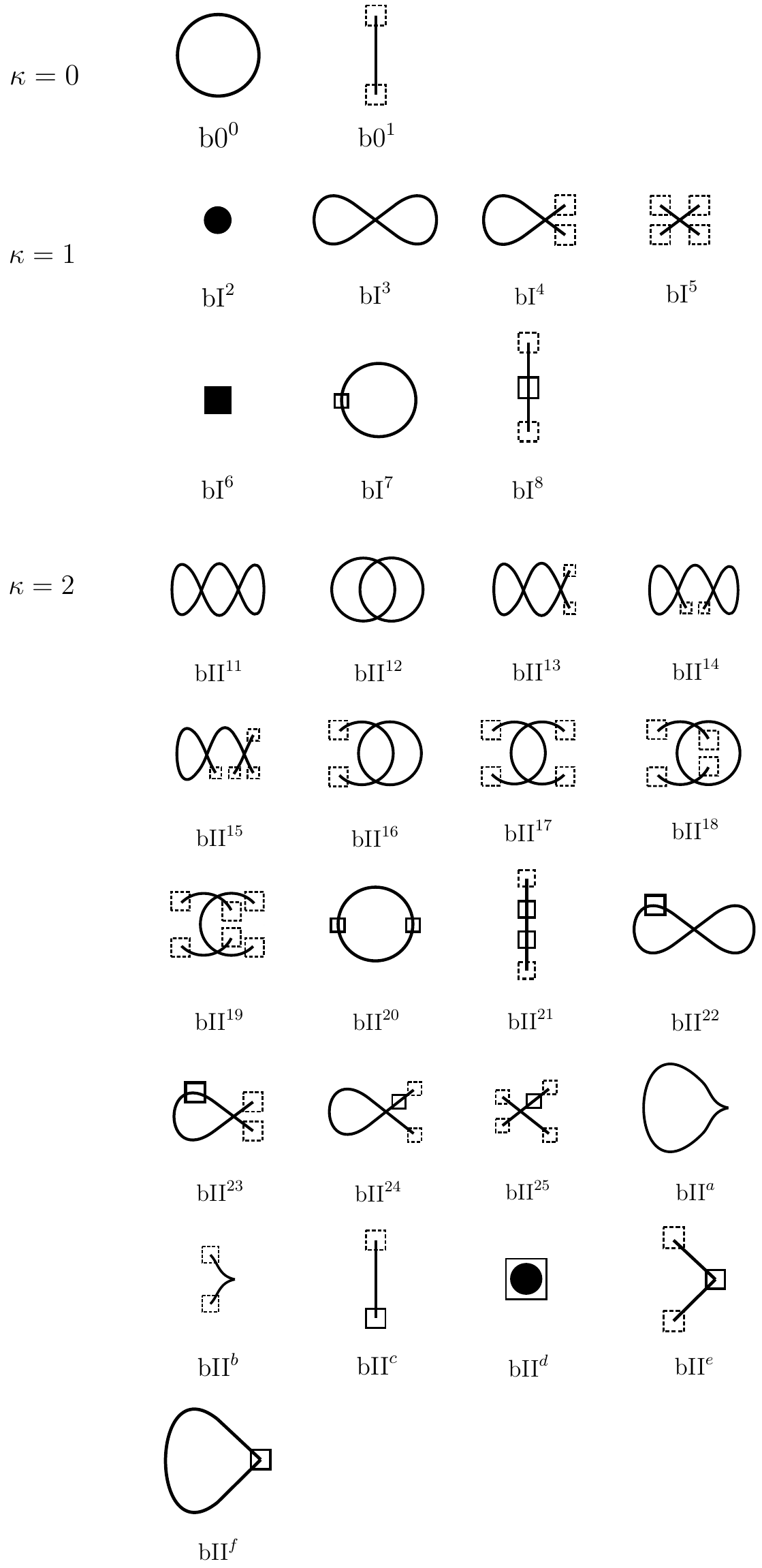}
    \caption{Fibers of stable maps of compact orientable $3$-manifolds with boundary into $\mathbb{R}^2$.}
    \label{Fig3}
\end{figure}

\begin{rem}
In this paper, we focus on a stable $\partial$-fold map from a compact orientable $3$-manifold to the plane without singular points. 
Such a map has only ${\rm b0}^0$, ${\rm b0}^1$, ${\rm bI}^6$, ${\rm bI}^7$, ${\rm bI}^8$, ${\rm bII}^{20}$, ${\rm bII}^{21}$, and ${\rm bII}^{\mu,\nu}$ for $\mu,\nu\in\{6,7,8\}$, as its fiber in~\Cref{Fig3}. 
\end{rem}

In \Cref{Fig3}, dots, crossings, and sharp points have the same meaning as above.
Squares with dashed outlines in ${\rm b0}^1$, ${\rm bI}^4$, ${\rm bI}^5$, ${\rm bI}^8$, ${\rm bII}^{13}$, ${\rm bII}^{14}$, ${\rm bII}^{15}$, ${\rm bII}^{16}$, ${\rm bII}^{17}$, ${\rm bII}^{18}$, ${\rm bII}^{19}$, ${\rm bII}^{21}$, ${\rm bII}^{23}$, ${\rm bII}^{24}$, ${\rm bII}^{25}$, ${\rm bII}^{b}$, ${\rm bII}^{c}$, and ${\rm bII}^{e}$ indicate intersections of the fibers with the boundary, while filled square in ${\rm bI}^6$ and open squares in ${\rm bI}^7$, ${\rm bI}8$, ${\rm bII}^{20}$, ${\rm bII}^{21}$, ${\rm bII}^{22}$, ${\rm bII}^{23}$, ${\rm bII}^{24}$, and ${\rm bII}^{25}$ represent boundary definite fold points and boundary indefinite fold points, respectively. 
The open square in ${\rm bII}^c$ represents a cusp point of
$F|_{\partial N}$.
The open square containing a dot in ${\rm bII}^d$ and the open
squares in ${\rm bII}^e$ and ${\rm bII}^f$ represent
$B_2$-singular points of $F$, whose types are distinguished by
the signs appearing in their local normal forms.
See \cite{SY2} for further details.

\begin{rem}
\Cref{Fig2} and \Cref{Fig3} contain the same pictures. 
However, we use different symbols to distinguish fibers of maps on closed manifolds from those of manifolds with boundary. 
Furthermore, unlike the notation used in \cite[Proposition~2.4]{SY2}, we use the symbols ${\rm bI}^{\mu}$, ${\rm bII}^{\mu\nu}$, and ${\rm bII}^{\mu,\nu}$, since we consider only orientable source manifolds. 
\end{rem}

To state our proposition, we introduce notation; see \cite{Sae6} for further details.
Let $M$ be a compact orientable $3$-manifold, possibly with boundary, and let $f\colon M\to\mathbb{R}^2$ be a stable map. 
Let $\mathcal{F}$ denote an equivalence class of fibers appearing in \Cref{Prop5.4} and \Cref{Prop5.5}. 
Suppose first that $M$ is closed.
We denote by $\mathcal{F}(f)$ the set of points $q\in\mathbb{R}^2$ such that the fiber of $f$ over $q$ is equivalent to the disjoint union of a representative of $\mathcal{F}$ and finitely many fibers of trivial circle bundles. 
If $M$ has non-empty boundary, we define $\mathcal{F}(f)$ to be the set of points $q\in\mathbb{R}^2$ such that the fiber of $f$ over $q$ is equivalent to the disjoint union of a representative of $\mathcal{F}$ and finitely many fibers of trivial circle bundles or trivial $I$-bundles. 
When $\mathcal{F}(f)$ is finite, we write $|\mathcal{F}(f)|$ for its cardinality.

Using this notation, we now state our proposition.
It concerns stable maps from compact orientable $3$-manifolds with boundary to the plane that have no singular points.
Such maps are considered in the next subsection.  
The following proposition describes how the fibers change under the doubling of maps.
Note that $D(F)$ is also stable, provided that $F$ has no singular points.
This follows from the classification of stable maps from $3$-manifolds to the plane; see, for example, \cite[Proposition~2.1]{SY1} and \cite{Lev3}.

\begin{prop}\label{Prop5.8}
Let $N$ be a compact orientable $3$-manifold with boundary, and let $F\colon N\to\mathbb{R}^2$ be a stable $\partial$-fold map such that $S(F)=\emptyset$.
Suppose that the fiber of $F$ over $q\in\mathbb{R}^2$ is equivalent to the disjoint union of the fiber depicted on the left-hand side of \Cref{Fig4} and finitely many copies of fibers of trivial circle bundles and trivial $I$-bundles. 
Then, the fiber of $D(F)$ over $q$ is equivalent to the disjoint union of the fiber depicted on the right-hand side of \Cref{Fig4} and finitely many copies of fibers of trivial circle bundles. 
Moreover, ${\rm bII}^{21}(F)={\rm II}^2(D(F))$ and ${\rm bII}^{20}(F)={\rm II}^3(D(F))$. 
\end{prop}

\begin{figure}
    \centering
    \includegraphics[width=1\linewidth]{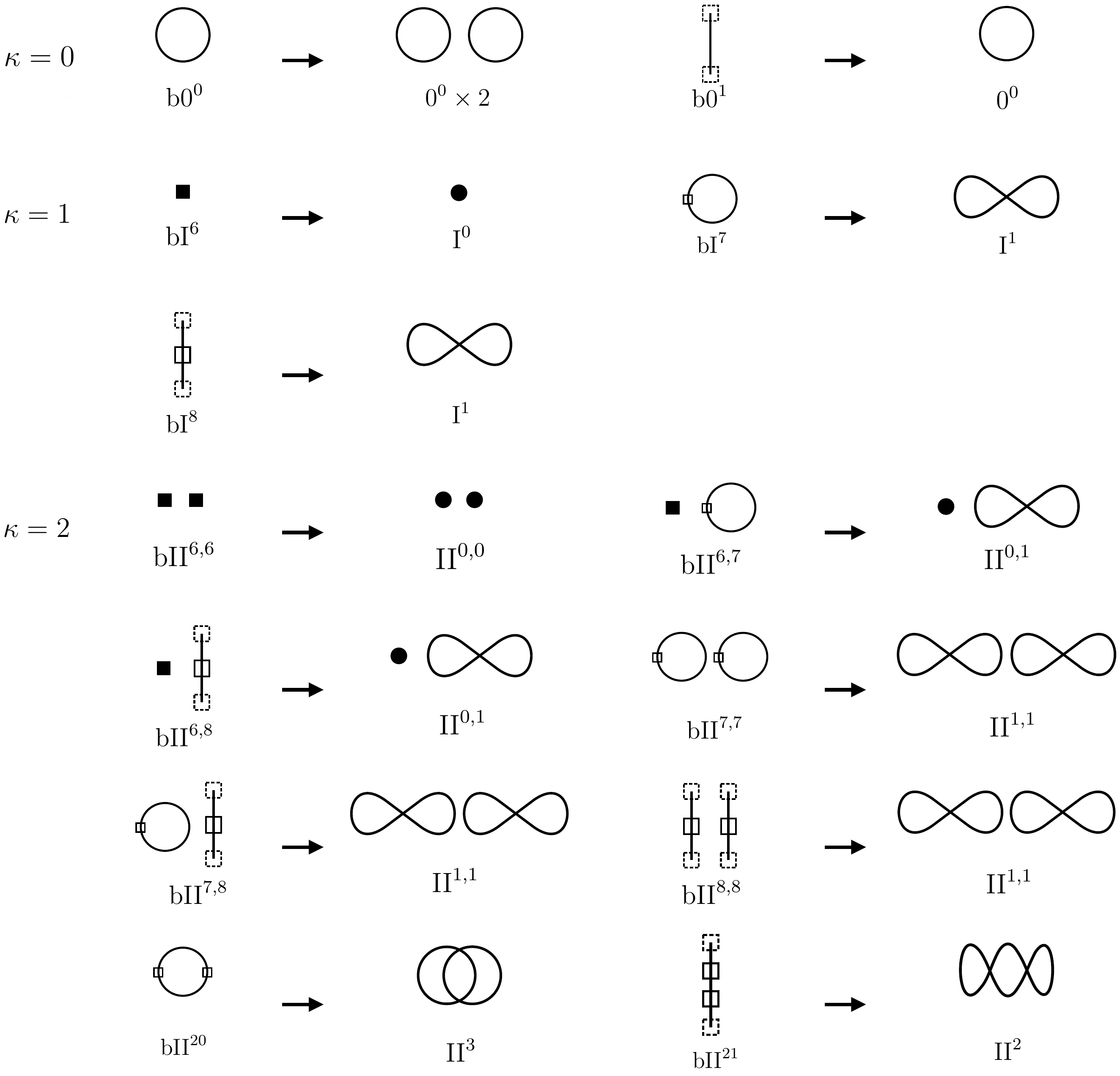}
    \caption{The changes of fibers by doubling.}
    \label{Fig4}
\end{figure}

\begin{proof}
We first prove the first assertion of the proposition.
By the definition of $D(F)$, the fiber of $D(F)$ over $q\in\mathbb{R}^2$ is obtained from two copies of the fiber of $F$ over $q$ by identifying the corresponding points lying in $\partial N$. 
Recall that $D(F)$ agrees with $F$ outside a collar neighborhood of $\partial N$.
Consequently, the portions of the fibers disjoint from $\partial N$ remain unchanged under doubling.
In \Cref{Fig3}, only the filled, unfilled squares, and dotted squares represent points at which the fiber intersects $\partial N$. 
The corresponding local changes under doubling are described in \Cref{Lem3.5}.
More precisely, identifying the two filled squares produces a single black dot, since a boundary definite fold point of $F|_{\partial N}$ gives rise to a definite fold point of $D(F)$.
Identifying the two unfilled squares produces a crossing, since a boundary indefinite fold point of $F|_{\partial N}$ gives an indefinite fold point of $D(F)$.
By contrast, identifying the two dotted squares produces no singular points, since a regular point of $F|_{\partial N}$ gives rise to a regular point of $D(F)$. 
These local changes are illustrated in~\Cref{Fig5}.

\begin{figure}
    \centering
    \includegraphics[width=0.6\linewidth]{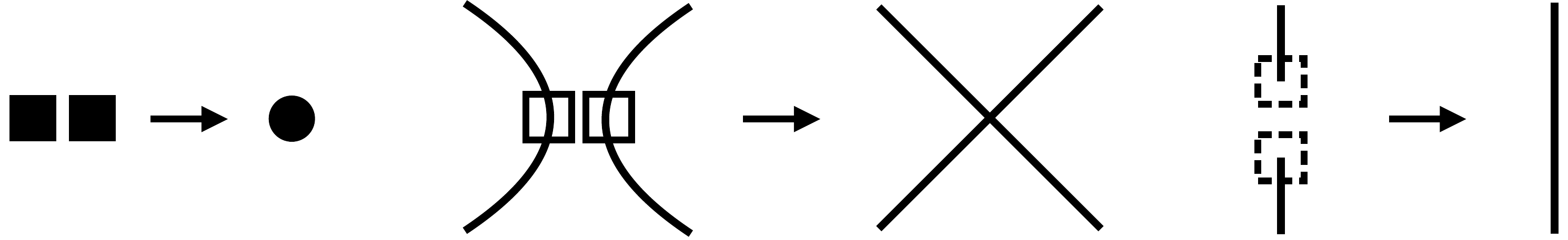}
    \caption{The local changes of fibers near the boundary.}
    \label{Fig5}
\end{figure}

In what follows, we show that each fiber depicted on the right-hand side of~\Cref{Fig4} is obtained from the corresponding fiber on the left-hand side by gluing described above.

We first consider the case $\kappa=0$. 
It follows immediately that two copies of ${\rm 0}^0$ are obtained from two copies of ${\rm b0}^0$ by gluing.
Moreover, since there is only one possible way to glue two copies of ${\rm b0}^1$, the resulting fiber is ${\rm 0}^0$.

We next consider the case $\kappa=1$. 
By \Cref{Lem3.5}, gluing two copies of ${\rm bI}^6$ produces ${\rm I}^0$. 
Moreover, since two boundary indefinite fold points give rise to a single indefinite fold point under gluing, as shown in \Cref{Lem3.5}, gluing two copies of ${\rm bI}^7$ produces a single fiber ${\rm I}^1$.
Here, the resulting one is determined by the orientability of $D(N)$.
By a similar argument, gluing two copies of ${\rm bI}^8$ also produces a single fiber ${\rm I}^1$.

\begin{figure}
    \centering
    \begin{subfigure}[b]{0.4\linewidth}
        \centering
        \includegraphics[width=\linewidth]{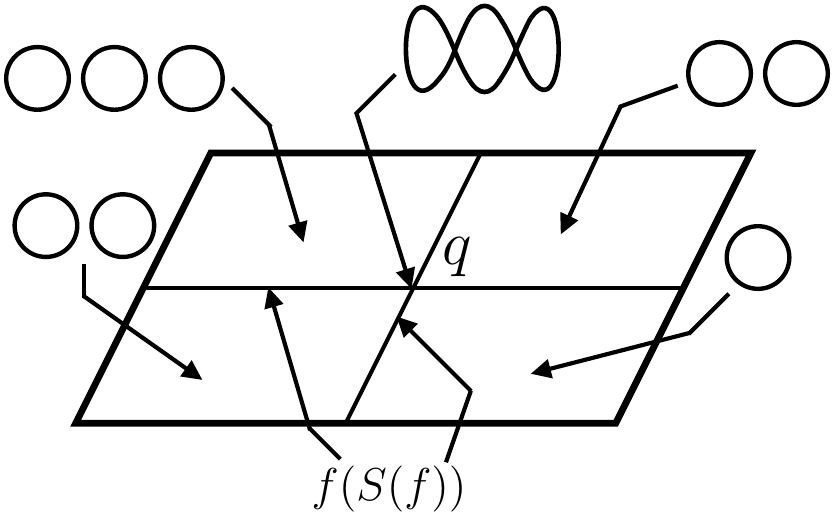}
        \caption{Type ${\rm II}^2$.}
        \label{Fig6(1)}
    \end{subfigure}%
    \hspace{0.05\linewidth}
    \begin{subfigure}[b]{0.4\linewidth}
        \centering
        \includegraphics[width=\linewidth]{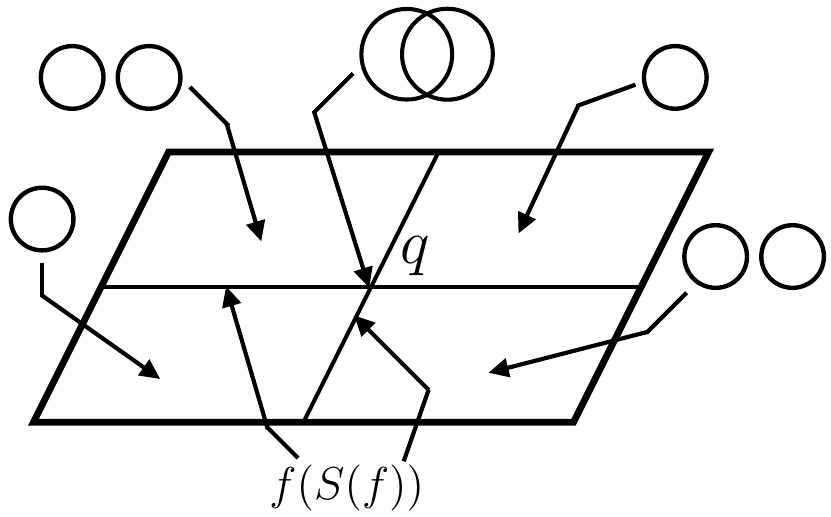}
        \caption{Type ${\rm II}^3$.}
        \label{Fig6(2)}
    \end{subfigure}
    \caption{The fibers around the singular fiber ${\rm II}^2$ and ${\rm II}^3$.}
    \label{Fig6}
\end{figure}

We finally consider the case $\kappa=2$. 
For fibers ${\rm bI}^{\mu,\nu}$, where $\mu, \nu\in\{6,7,8\}$, the assertion follows immediately from the case $\kappa=1$. 
It remains to consider the fibers ${\rm bII}^{20}$ and ${\rm bII}^{21}$.
We begin with an observation.
Since the fibers obtained by gluing two copies of ${\rm bII}^{20}$ or ${\rm bII}^{21}$ are connected and $D(F)$ is a stable fold map, it must be ${\rm II}^2$ or ${\rm II}^3$.
The two fibers are distinguished by the behavior of the nearby fibers around the corresponding double point; see~\Cref{Fig6}, where $q$ is a singular value of a stable map of closed oriented $3$-manifold into $\mathbb{R}^2$.
In particular, they can be distinguished by comparing the numbers of connected components derived from the singular fibers.
We first identify the corresponding pairs of dotted squares in the two copies of ${\rm bII}^{20}$, respectively ${\rm bII}^{21}$, and then consider the possible identifications of the remaining branches containing the indefinite fold points.
Once the corresponding dotted squares are fixed, this verification is purely combinatorial. 
Although there are many such identifications, each case can be decided by the same criterion: we first determine whether the identification is compatible with the orientability of $D(N)$, and, if it is admissible, we compare the nearby regular fibers
around the resulting singular fiber.
As explained above, the fibers ${\rm II}^2$ and ${\rm II}^3$ are distinguished by the numbers of connected components of these nearby fibers.
\Cref{Fig7} illustrates a representative admissible case and an inadmissible case for ${\rm bII}^{20}$.
Applying the same criterion to all remaining identifications shows that every admissible identification of two copies of ${\rm bII}^{20}$ produces a fiber of type ${\rm II}^3$.
Similarly, every admissible identification of two copies of ${\rm bII}^{21}$ produces a fiber of type ${\rm II}^2$.
This proves the first assertion of the proposition.
The second one follows immediately from the first.
\end{proof}

\begin{figure}
    \centering
    \begin{subfigure}[b]{0.19\linewidth}
        \centering
        \includegraphics[width=\linewidth]{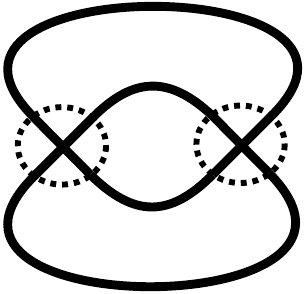}
        \caption{Admissible.}
        \label{Fig7(1)}
    \end{subfigure}
    \hspace{30pt}
    \begin{subfigure}[b]{0.19\linewidth}
        \centering
        \includegraphics[width=\linewidth]{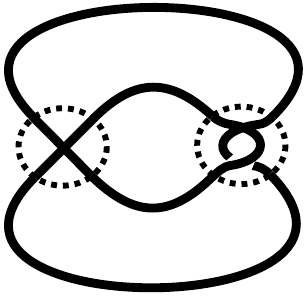}
        \caption{Inadmissible.}
        \label{Fig7(2)}
    \end{subfigure}
    \caption{Two possible cases from ${\rm bII}^{20}$, where the dotted circles represent the variable parts.}
    \label{Fig7}
\end{figure}

\subsection{Singular fibers and simplicial volumes of $3$-manifolds with boundary}
In this subsection, by combining the results of Ishikawa--Koda~\cite{IK} with \Cref{Prop5.8}, we establish an inequality relating the singular fibers of stable maps from compact oriented $3$-manifolds with boundary to the plane to the simplicial volumes of their source manifolds.

We first recall the definition of simplicial volumes.
For more details, see~\cite{Gro}.

\begin{defi}
Let $M$ be a compact oriented $m$-dimensional manifold possibly with boundary, and let $[M,\partial M]\in H_m(M,\partial M;\mathbb{R})$ denote the fundamental class.
Then, 
$$
\lVert M,\partial M\rVert
=
\inf\Big\{
\sum_{r=1}^{s}|a_r|
\mid
[M,\partial M]
\text{ is represented by }
\sum_{r=1}^{s}a_r\sigma_r\in C_{m}(M,\partial M)
\Big\}
$$
is called the \emph{simplicial volume} of $M$.
\end{defi}

We next introduce the result by Ishikawa--Koda~\cite{IK}.
To introduce it, we recall the \emph{stable map complexity} of closed oriented $3$-manifold $M$ as follows:
$$
{\rm smc}(M)=
\min_{f}
\{
|{\rm II}^2(f)|+2|{\rm II}^3(f)|
\},
$$
where $f\colon M\to\mathbb{R}^2$ is any stable fold map.
The theorem by Ishikawa--Koda~\cite{IK} gives the relation between the stable map complexity of $M$ and the simplicial volume of $M$.

\begin{thm}[Ishikawa--Koda]
Let $M$ be a closed oriented $3$-manifold. 
Then,
$$
\lVert M\rVert V_{\text{tet}}
\leq
2 {\rm smc}(M)V_{\text{oct}}, 
$$
where $\lVert M\rVert=\lVert M,\emptyset\rVert$, and $V_{\text{tet}}=1.01…$ and $V_{\text{oct}}=3.66…$ denote the volumes of the ideal regular tetrahedron and the ideal regular octahedron, respectively. 
\end{thm}

By analogy with the stable map complexity above, we define \emph{$\partial$-stable map complexity} for compact connected oriented $3$-manifolds with boundary as follows:
$$
{\rm smc}^{\partial}(N)
=
\min_{F}\{
|{\rm bII}^{21}(F)|+2|{\rm bII}^{20}(F)|
\}
$$
where $F\colon N\to\mathbb{R}^2$ is any stable $\partial$-fold map with $S(F)=\emptyset$.

\begin{rem}\label{Rem5.11}
Any compact connected orientable $3$-manifold with boundary can be immersed into $\mathbb{R}^3$ by Whitehead~\cite{Whi}.
Shibata~\cite{Shi} indicates that the composition of this immersion and the projection $\pi\colon\mathbb{R}^3\to\mathbb{R}^2$ produces a stable map $F\colon N\to\mathbb{R}^2$ with $S(F)=\emptyset$.
Moreover, since $\partial N$ is orientable, the cusps of $F|_{\partial N}$ can be eliminated up to homotopy by Yamamoto~\cite{Yam} and Levine~\cite{Lev2}, while preserving the map as a submersion.
Therefore, for any compact connected orientable $3$-manifold with boundary, $\partial$-stable map complexity is defined as a finite non-negative integer.
\end{rem}

With these preparations, we obtain the following theorem.

\begin{thm}\label{Thm5.12}
Let $N$ be a compact connected oriented $3$-manifold with toroidal boundary.
Then, 
$$
\lVert N,\partial N\rVert V_{\text{tet}}\leq {\rm smc}^{\partial}(N)V_{\text{oct}}, 
$$
where $V_{\text{tet}}$ and $V_{\text{oct}}$ are as above. 
\end{thm}

\begin{proof}
By Ishikawa--Koda~\cite{IK}, we have the inequality 
$$
\lVert D(N)\rVert V_{\text{tet}}\leq 2{\rm smc}(D(N))V_{\text{oct}}.
$$

We first consider the term on the right-hand side.
Let $G\colon N\to\mathbb{R}^2$ be any stable $\partial$-fold map satisfying $S(G)=\emptyset$.
By the characterization of stable maps from closed orientable $3$-manifolds, $D(G)$ is also a stable fold map.
Therefore, by the definition of stable map complexity, we have
$$
{\rm smc}(D(N))
\leq 
\min_{G}\{|{\rm II}^2(D(G))|+2|{\rm II}^3(D(G))|\},
$$
where the minimum is taken over all stable $\partial$-fold maps $G\colon N\to\mathbb{R}^2$ with $S(G)=\emptyset$. 
Furthermore, by \Cref{Prop5.8}, ${\rm bII}^{20}(G)={\rm II}^{3}(D(G))$ and ${\rm bII}^{21}(G)={\rm II}^{2}(D(G))$ hold. 
Therefore, we obtain
$$
\min_{G}\{|{\rm II}^{2}(D(G))|+2|{\rm II}^{3}(D(G))|\}
=
\min_{G}\{|{\rm bII}^{21}(G)|+2|{\rm bII}^{20}(G)|\}, 
$$ 
and the right-hand side is $\mathrm{smc}^{\partial}(N)$.

Furthermore, since $\partial N$ is a disjoint union of tori, it follows from \cite{Kue} that $\lVert D(N) \rVert =2\lVert N,\partial N\rVert $. 
Combining these equalities and inequalities, we obtain the desired one. 
\end{proof}

\begin{rem}
In this section, we introduce ${\rm smc}^{\partial}(N)$ as the complexity of a compact connected oriented $3$-manifold with boundary.
One may also define a similar complexity by allowing maps with singular points.
However, it is currently unclear whether such a complexity can be used to sharpen the inequality above. 
\end{rem}

The following examples illustrate the lower bounds on singular fibers obtained from \Cref{Thm5.12}.

\begin{example}
Let $N$ be the exterior of the knot $4_1$, which is also called a figure-eight knot.
It is known that ${\rm vol}(N)= 2.02\dots$; see, for example, \cite{AHW}. 
Therefore, by \Cref{Thm5.12} and \Cref{Thm5.17} described later, every stable $\partial$-fold map $F\colon N\to\mathbb{R}^2$ satisfying $S(F)=\emptyset$ must satisfy
$$
|{\rm bII}^{21}(F)|+2|{\rm bII}^{20}(F)|\geq 1.
$$
This implies that $F$ must have at least one singular fiber of ${\rm bII}^{20}$ or ${\rm bII}^{21}$.
\end{example}

\begin{example}
Let $N$ be the exterior of the knot $6_2$.
It is known that ${\rm vol}(N)=4.40\dots$; see, for example, \cite{AHW}.
Therefore, by \Cref{Thm5.12} and \Cref{Thm5.17} described later, every stable $\partial$-fold map $F\colon N\to\mathbb{R}^2$ satisfying $S(F)=\emptyset$ must satisfy 
$$
|{\rm bII}^{21}(F)|+2|{\rm bII}^{20}(F)|\geq 2.
$$
This implies that $F$ must have either at least one ${\rm bII}^{20}$ or at least two ${\rm bII}^{21}$.
\end{example}

\subsection{Simple stable map and graph manifolds}
In this subsection, using \Cref{Thm5.12}, we study simple stable maps from compact orientable $3$-manifolds with boundary to the plane with certain conditions.
More precisely, we show that if such a map has no singular points in the interior, then its source manifold is a graph manifold.

We begin by introducing the notion of a simple map.

\begin{defi}
Let $N$ be an $n$-dimensional manifold with boundary, and let $L$ be an $l$-dimensional manifold, where $n>l$. 
A submersion $F\colon N\to L$ is said to be \emph{simple} if, for every singular value $q$ of $F|_{\partial N}$, the fiber $F^{-1}(q)$ contains at most one singular point of $F|_{\partial N}$. 
\end{defi}

We now recall some basic facts about $3$-manifolds.
Every compact irreducible orientable $3$-manifold $N$ whose boundary is toroidal admits a decomposition called the JSJ decomposition.
More precisely, by cutting the manifold along a finite collection of pairwise disjoint tori with certain conditions, it can be decomposed into pieces that are either Seifert fibered or hyperbolic.
The hyperbolic volume ${\rm vol}(N)$ of $N$ is defined to be the sum of the hyperbolic volumes of all the hyperbolic pieces appearing in its JSJ decomposition. 
The following theorem follows from a result of Soma~\cite{Som}; see \cite[Theorem~1]{FLMQ} for a statement.

\begin{thm}[Soma]\label{Thm5.17}
Let $N$ be a compact irreducible oriented $3$-manifold with toroidal boundary.
Then, 
$$
\lVert N,\partial N\rVert =\frac{{\rm vol}(N)}{V_{\text{tet}}}.
$$
\end{thm}

Following \cite{Neu}, we call a compact oriented $ 3$-manifold with toroidal boundary a \emph{graph manifold} if its JSJ decomposition contains no hyperbolic pieces.
As an application of \Cref{Thm5.12}, we obtain the following corollary.

\begin{cor}\label{Cor5.18}
Let $N$ be a compact connected irreducible oriented $3$-manifold with toroidal boundary.
If $N$ admits a simple stable $\partial$-fold map into $\mathbb{R}^2$ without singular points, then $N$ is a graph manifold. 
\end{cor}

\begin{proof}
Let $F\colon N\to\mathbb{R}^2$ be a simple stable $\partial$-fold map with $S(F)=\emptyset$. 
Then, we have $|{\rm bII}^{20}(F)|=|{\rm bII}^{21}(F)|=0$.
It follows from \Cref{Thm5.12} that $\lVert N,\partial N\rVert=0$. 
By \Cref{Thm5.17}, we further obtain ${\rm vol}(N)=0$. 
Therefore, the JSJ decomposition of $N$ contains no hyperbolic pieces, and hence $N$ is a graph manifold. 
\end{proof}

\section{Results for existence and non-existence of $\partial$-fold maps}
In this section, we give results for existence and non-existence of $\partial$-fold maps on manifolds with boundary. 
We further apply these results to $\partial$-fold maps from punctured complex projective spaces to Euclidean spaces.

\subsection{An existence result and its application}
In this subsection, we present an existence result for $\partial$-fold maps. 
Then, we apply this result to determine pairs $(m,l)$ for which $\mathbb{C}P^{m}\setminus\Int D^{2m}$ admits a $\partial$-fold map into $\mathbb{R}^l$. 
The following proposition is obtained by removing a sufficiently small open ball centered at a regular point of a fold map on a closed manifold.

\begin{prop}
Let $M$ be a closed $m$-dimensional manifold, and let $L$ be an $l$-dimensional manifold, where $m>l$. 
If $M$ admits a fold map into $L$, then $M\setminus\Int D^m$ admits a $\partial$-fold map into $L$.
\end{prop}

We next give the following proposition concerning the existence of fold maps from $\mathbb{C}P^{m}$ to Euclidean spaces. 
The proof of the proposition is given in the Appendix.

\begin{prop}
Consider the pairs $(m,l)$ satisfying $1\leq m,l\leq 6$ and $2m>l$ listed in \Cref{Fig8}. 
A blue circle indicates that $\mathbb{C}P^{m}$ admits a fold map into $\mathbb{R}^l$.
A red cross indicates that no such fold map exists.
A green triangle indicates that, to the best of the author’s knowledge, the existence of such a fold map remains unknown.
\end{prop}

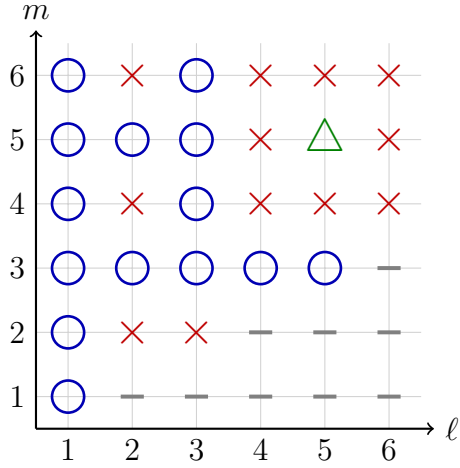
\begin{figure}[t]
\centering
\begin{tikzpicture}[scale=0.85]

\draw[step=1, gray!35, thin] (0.5,0.5) grid (6.5,6.5);

\draw[->, thick] (0.5,0.5) -- (6.7,0.5)
  node[right] {$\ell$};
\draw[->, thick] (0.5,0.5) -- (0.5,6.7)
  node[above] {$m$};

\foreach \x in {1,...,6}
  \node[below] at (\x,0.5) {\(\x\)};

\foreach \y in {1,...,6}
  \node[left] at (0.5,\y) {\(\y\)};

%l=1
\draw[blue!70!black, line width=1pt] (1,1) circle[radius=0.25];
\draw[blue!70!black, line width=1pt] (1,2) circle[radius=0.25];
\draw[blue!70!black, line width=1pt] (1,3) circle[radius=0.25];
\draw[blue!70!black, line width=1pt] (1,4) circle[radius=0.25];
\draw[blue!70!black, line width=1pt] (1,5) circle[radius=0.25];
\draw[blue!70!black, line width=1pt] (1,6) circle[radius=0.25];

%l=2
\draw[gray!100, line width=1.5pt] (1.82,1) -- (2.18,1);
\node[red!75!black]   at (2,2) {\Large$\times$};
\draw[blue!70!black, line width=1pt] (2,3) circle[radius=0.25];
\node[red!75!black]   at (2,4) {\Large$\times$};
\draw[blue!70!black, line width=1pt] (2,5) circle[radius=0.25];
\node[red!75!black]   at (2,6) {\Large$\times$};

%l=3
\draw[gray!100, line width=1.5pt] (2.82,1) -- (3.18,1);
\node[red!75!black]   at (3,2) {\Large$\times$};
\draw[blue!70!black, line width=1pt] (3,3) circle[radius=0.25];
\draw[blue!70!black, line width=1pt] (3,4) circle[radius=0.25];
\draw[blue!70!black, line width=1pt] (3,5) circle[radius=0.25];
\draw[blue!70!black, line width=1pt] (3,6) circle[radius=0.25];

%l=4
\draw[gray!100, line width=1.5pt] (3.82,1) -- (4.18,1);
\draw[gray!100, line width=1.5pt] (3.82,2) -- (4.18,2);
\draw[blue!70!black, line width=1pt] (4,3) circle[radius=0.25];
\node[red!75!black]   at (4,4) {\Large$\times$};
\node[red!75!black]   at (4,5) {\Large$\times$};
\node[red!75!black]   at (4,6) {\Large$\times$};

%l=5
\draw[gray!100, line width=1.5pt] (4.82,1) -- (5.18,1);
\draw[gray!100, line width=1.5pt] (4.82,2) -- (5.18,2);
\draw[blue!70!black, line width=1pt] (5,3) circle[radius=0.25];
\node[red!75!black]   at (5,4) {\Large$\times$};
\node[green!50!black] at (5,5) {\Large$\triangle$};
\node[red!75!black]   at (5,6) {\Large$\times$};

%l=6
\draw[gray!100, line width=1.5pt] (5.82,1) -- (6.18,1);
\draw[gray!100, line width=1.5pt] (5.82,2) -- (6.18,2);
\draw[gray!100, line width=1.5pt] (5.82,3) -- (6.18,3);
\node[red!75!black]   at (6,4) {\Large$\times$};
\node[red!75!black]   at (6,5) {\Large$\times$};
\node[red!75!black]   at (6,6) {\Large$\times$};

\end{tikzpicture}

\caption{
For fold maps from $\mathbb{C}P^{m}$ to $\mathbb{R}^l$. 
A circle, a cross, and a triangle indicate existence,
non-existence, and an open case, respectively.
A thick dash indicates a case not considered in this paper.
}
\label{Fig8}
\end{figure}

The following corollary follows immediately from the preceding propositions.

\begin{cor}\label{Cor6.3}
For every pair $(m,l)$ with a blue circle in \Cref{Fig8}, $\mathbb{C}P^{m}\setminus\Int D^{2m}$ admits a $\partial$-fold map into $\mathbb{R}^l$. 
\end{cor}

\subsection{A non-existence result and its application}

In this subsection, we present a non-existence result for $\partial$-fold maps obtained by gluing maps. 
Then, we apply this result to determine pairs $(m,l)$ for which there exists no $\partial$-fold map from $\mathbb{C}P^{m}\setminus\Int D^{2m}$ to $\mathbb{R}^l$. 
The following proposition is an immediate consequence of \Cref{Lem3.5}, applied to $\partial$-fold maps.

\begin{prop}
Let $N$ be a compact oriented $n$-dimensional manifold with boundary, and let $L$ be an $l$-dimensional manifold, where $n>l$. 
If $N$ admits a $\partial$-fold map into $L$, then $N\cup_{{\rm id}}\overline{N}$ admits a fold map into $L$.
\end{prop}

We further give the following proposition concerning the existence and non-existence of fold maps from $\mathbb{C}P^{m}\#\overline{\mathbb{C}P^{m}}$ to $\mathbb{R}^l$. 
The proof of the proposition is given in the Appendix.

\begin{prop}
Consider the pairs $(m,l)$ satisfying $1\leq m,l\leq 6$ and $2m>l$ listed in \Cref{Fig9}. 
A blue circle indicates that $\mathbb{C}P^{m}\#\overline{\mathbb{C}P^{m}}$ admits a fold map into $\mathbb{R}^l$, whereas a red cross indicates that no such fold map exists.
A green triangle indicates that, to the best of the author’s knowledge, the existence of such a fold map remains unknown. 
\end{prop}

\begin{figure}[t]
\centering
\begin{tikzpicture}[scale=0.85]

\draw[step=1, gray!35, thin] (0.5,0.5) grid (6.5,6.5);

\draw[->, thick] (0.5,0.5) -- (6.7,0.5)
  node[right] {$\ell$};
\draw[->, thick] (0.5,0.5) -- (0.5,6.7)
  node[above] {$m$};

\foreach \x in {1,...,6}
  \node[below] at (\x,0.5) {\(\x\)};

\foreach \y in {1,...,6}
  \node[left] at (0.5,\y) {\(\y\)};

%l=1
\draw[blue!70!black, line width=1pt] (1,1) circle[radius=0.25];
\draw[blue!70!black, line width=1pt] (1,2) circle[radius=0.25];
\draw[blue!70!black, line width=1pt] (1,3) circle[radius=0.25];
\draw[blue!70!black, line width=1pt] (1,4) circle[radius=0.25];
\draw[blue!70!black, line width=1pt] (1,5) circle[radius=0.25];
\draw[blue!70!black, line width=1pt] (1,6) circle[radius=0.25];

%l=2
\draw[gray!100, line width=1.5pt] (1.82,1) -- (2.18,1);
\draw[blue!70!black, line width=1pt] (2,2) circle[radius=0.25];
\draw[blue!70!black, line width=1pt] (2,3) circle[radius=0.25];
\draw[blue!70!black, line width=1pt] (2,4) circle[radius=0.25];
\draw[blue!70!black, line width=1pt] (2,5) circle[radius=0.25];
\draw[blue!70!black, line width=1pt] (2,6) circle[radius=0.25];

%l=3
\draw[gray!100, line width=1.5pt] (2.82,1) -- (3.18,1);
\draw[blue!70!black, line width=1pt] (3,2) circle[radius=0.25];
\draw[blue!70!black, line width=1pt] (3,3) circle[radius=0.25];
\draw[blue!70!black, line width=1pt] (3,4) circle[radius=0.25];
\draw[blue!70!black, line width=1pt] (3,5) circle[radius=0.25];
\draw[blue!70!black, line width=1pt] (3,6) circle[radius=0.25];

%l=4
\draw[gray!100, line width=1.5pt] (3.82,1) -- (4.18,1);
\draw[gray!100, line width=1.5pt] (3.82,2) -- (4.18,2);
\node[green!50!black] at (4,3) {\Large$\triangle$};
\node[green!50!black] at (4,4) {\Large$\triangle$};
\node[red!75!black]   at (4,5) {\Large$\times$};
\node[red!75!black]   at (4,6) {\Large$\times$};

%l=5
\draw[gray!100, line width=1.5pt] (4.82,1) -- (5.18,1);
\draw[gray!100, line width=1.5pt] (4.82,2) -- (5.18,2);
\node[green!50!black] at (5,3) {\Large$\triangle$};
\node[green!50!black] at (5,4) {\Large$\triangle$};
\node[green!50!black] at (5,5) {\Large$\triangle$};
\node[green!50!black] at (5,6) {\Large$\triangle$};

%l=6
\draw[gray!100, line width=1.5pt] (5.82,1) -- (6.18,1);
\draw[gray!100, line width=1.5pt] (5.82,2) -- (6.18,2);
\draw[gray!100, line width=1.5pt] (5.82,3) -- (6.18,3);
\node[green!50!black] at (6,4) {\Large$\triangle$};
\node[red!75!black]   at (6,5) {\Large$\times$};
\node[red!75!black]   at (6,6) {\Large$\times$};

\end{tikzpicture}

\caption{
For fold maps from $\mathbb{C}P^m\sharp\overline{\mathbb{C}P^m}$ to $\mathbb{R}^l$. 
A circle, a cross, and a triangle indicate existence,
non-existence, and an open case, respectively.
A thick dash indicates a case not considered in this paper.
}
\label{Fig9}
\end{figure}

The following corollary follows immediately from the preceding two propositions.

\begin{cor}\label{Cor6.6}
For every pair $(m,l)$ marked with a red cross in \Cref{Fig9}, $\mathbb{C}P^{m}\setminus\Int D^{2m}$ does not admit a $\partial$-fold map into $\mathbb{R}^l$. 
\end{cor}

Combining \Cref{Cor6.3} and \Cref{Cor6.6} yields the existence and non-existence results summarized in \Cref{Fig1} for $\partial$-fold maps from $\mathbb{C}P^{m}\setminus\Int D^{2m}$ to $\mathbb{R}^l$.

\subsection{Remarks on unsolved cases}
In this subsection, we make several remarks on the cases in \Cref{Fig1} for which the existence problem for $\partial$-fold maps remains unsolved. 
We give a necessary condition for the existence of a $\partial$-fold map from $\mathbb{C}P^{2}\setminus\Int D^{4}$ to $\mathbb{R}^3$. 
This condition imposes a restriction on the singular point set of the map on the boundary.

\begin{cor}
Assume that there exists a $\partial$-fold map $F\colon\mathbb{C}P^{2}\setminus\Int D^{4}\to\mathbb{R}^3$.
Then, 
$$
\chi({}^{\partial\!}A_{1}^{+}(F))\equiv\chi({}^{\partial\!}A_{1}^{-}(F)) \pmod 4. 
$$
\end{cor}

\begin{proof}
We equip $\mathbb{C}P^2\setminus\Int D^4$ with the orientation induced by the natural orientation of $\mathbb{C}P^2$.
Let $S_d$ denote the submanifold of $S(F)$ consisting of definite fold points.
By \Cref{Lem3.1} and~\cite[Corollary~2.7]{Sae1}, we have $S(F)\cdot S(F)=S_d\cdot S_d$.
Moreover, $S_d$ is a closed orientable surface by~\cite[Lemma~2.8]{Sae1}, so we fix an orientation on each of its connected components.
Choose a generator $x$ of $H_2(\mathbb{C}P^2\setminus\Int D^4;\mathbb{Z})\cong\mathbb{Z}$ such that $x\cdot x=1$.
Then there exists an integer $k$ such that $[S_d]=kx$.
Therefore, $S(F)\cdot S(F)=[S_d]\cdot[S_d]=k^2$.
Since $\sigma(\mathbb{C}P^2\setminus\Int D^4)=1$, \Cref{Thm4.8} yields
$$
k^2-1
\equiv
\chi\bigl({}^{\partial\!}A_1^+(F)\bigr)
-
\chi\bigl({}^{\partial\!}A_1^-(F)\bigr)
\pmod{4}.
$$
On the other hand, the right-hand side is even by \Cref{Thm4.3}.
Hence, $k$ must be odd.
Since $k^2\equiv1\pmod{4}$ for every odd integer $k$, we conclude that
$$
\chi\bigl({}^{\partial\!}A_1^+(F)\bigr)
-
\chi\bigl({}^{\partial\!}A_1^-(F)\bigr)
\equiv0\pmod{4}, 
$$
which is the desired one. 
\end{proof}

We do not know whether $\partial$-fold maps exist in the remaining cases marked as unknown in \Cref{Fig1}.
In particular, to the best of the author’s knowledge, it remains unknown whether $\mathbb{C}P^{2}\setminus\Int D^{4}$ admits a $\partial$-fold map into $\mathbb{R}^3$ in general.

\section{Results on non-singular extensions}
In this section, we apply the results developed above to the non-singular extension problem. 
We first derive obstructions of Euler characteristic from \Cref{Thm4.3} and then introduce an extension-independent invariant for maps from closed oriented $3$-manifolds to $3$-manifolds.

\subsection{Definition of extensions and non-singular extensions}
In this subsection, we define extensions and non-singular extensions.

\begin{defi}\label{Defi7.1}
Let $M$ be a closed $m$-dimensional manifold, let $L$ be an $l$-dimensional manifold, and let $g\colon M\times[0,1)\to L$ be a submersion, where $m\geq l$.
Suppose that there exist a compact manifold $N$ with $\partial N=M$ and a map $F\colon N\to L$ that make the following diagram commutative: 
$$
\xymatrix{
M\times[0,1) \ar[r]^{~~~~~g} \ar[d]_{i} & L \\ 
N \ar[ur]_{F}. 
}
$$
Then, $F$ is called an \emph{extension of $g$}. 
In particular, if $F$ is a submersion, it is called a \emph{non-singular extension of $g$}.
Here, $i$ is a collar neighborhood of $\partial N$.
\end{defi}

\begin{rem}
\Cref{Defi7.1} may be regarded as a refinement of the definition of non-singular extensions on~\cite{Iwa3, Iwa4}, in which a non-singular extension of the map on the boundary to a collar neighborhood is prescribed in advance.
Such a definition is used in \cite{BL, Cur, Sei, Iwa1, Iwa2}. 
\end{rem}

Throughout this section, we restrict our attention to extensions that are $\partial$-Morin maps and to Morin maps as the maps to be extended.
Accordingly, unless otherwise stated, an extension is understood to be a $\partial$-Morin map.

\subsection{Formulas for non-singular extensions of fold maps}
In this subsection, we derive a necessary condition for the existence of a non-singular extension.
The following result is an immediate consequence of \Cref{Thm4.5}.

\begin{cor}
Let $M$ be a closed $m$-dimensional manifold, let $L$ be a connected $1$-dimensional manifold, and let $g\colon M\times[0,1)\to L$ be a submersion such that $g|_{M\times\{0\}}$ is a Morse function.
If $m$ is odd and $g$ admits a non-singular extension on $N$, then
$$
\chi(N)=\frac{1}{2}(\#{}^{\partial\!}A_1^{+}(g)- \#{}^{\partial\!}A_1^{-}(g)), 
$$
where ${}^{\partial\!}A_1^{\pm}(g)$ are defined similarly to those in \Cref{Thm4.5}, and $\#$ is the number of the connected components of the relevant set.
\end{cor}

Furthermore, the following result follows from \Cref{Thm4.5}.

\begin{cor}
Let $M$ be a closed $m$-dimensional manifold, let $L$ be a connected $l$-dimensional manifold, and let $g\colon M\times[0,1)\to L$ be a submersion whose restriction to $M\times\{0\}$ is a fold map such that $S(g|_{M\times\{0\}})$ is homeomorphic to a finite disjoint union of $(l-1)$-dimensional spheres.
If both $m$ and $l$ are odd and $l\geq 3$, and $g$ admits a non-singular extension on $N$, then 
$$
\chi(N)=\#{}^{\partial\!}A_1^+(g)-\#{}^{\partial\!}A_1^-(g), 
$$
where ${}^{\partial\!}A_1^{\pm}(g)$ are defined similarly to those in \Cref{Thm4.5}.
\end{cor}

Examples of fold maps whose singular point sets are homeomorphic to finite disjoint unions of spheres include special generic maps from simply connected closed manifolds to $\mathbb{R}^3$ \cite{Sak} and round fold maps \cite{Kit}.
We give a way to use the preceding formula.

\begin{example}
Let $i\colon S^3\to\mathbb{R}^4$ be the standard inclusion, and let $\pi\colon\mathbb{R}^4\to\mathbb{R}^3$ be the standard projection.
Define $f=\pi\circ i\colon S^3\to\mathbb{R}^3$.
Then, $f$ is a fold map whose singular point set is diffeomorphic to $S^2$. 
We extend $f$ over $S^3\times[0,1)$ in the direction of the unbounded
component of $\mathbb{R}^4\setminus S^3$ and denote the resulting submersion by $g$.
Then, $g$ is a submersion satisfying $g|_{S^3\times\{0\}}=f$.
With our convention, we have
$$
\#\,{}^{\partial\!}A_1^+(g)
-
\#\,{}^{\partial\!}A_1^-(g)
=-1.
$$
Therefore, for example, it follows that $g$ does not admit a non-singular extension to
$\mathbb{C}P^2\setminus\operatorname{Int}D^4$, since $\chi\bigl(\mathbb{C}P^2\setminus\operatorname{Int}D^4\bigr)=2$.
\end{example}

\subsection{An extension-independent invariant and an application}
We next show that the existence of a non-singular extension imposes constraints on all other extensions of the same prescribed map. 
For this purpose, we introduce the following extension-independent invariant.

\begin{lem}\label{Lem7.6}
Let $g\colon M\times[0,1)\to L$ be a submersion such that $g|_{M\times\{0\}}$ is a Morin map from a closed oriented $3$-manifold $M\times\{0\}$ to an orientable $3$-manifold $L$. 
Suppose that $g$ admits an extension $F$ of a compact oriented $4$-manifold $N$.
Then, the integer
$$
3\sigma(N)-S(F)\cdot S(F)
$$
depends only on $g$. 
We denote this integer by $I(g)$.
\end{lem}

\begin{proof}
Let $F_1\colon N_1\to L$ and $F_2\colon N_2\to L$ be extensions of $g$.
Gluing them along their common boundary via the identity map yields a map $F_1\cup F_2\colon N_1\cup_{\rm id}\overline{N_2}\to L$.
After a sufficiently small perturbation, we may assume that this map is stable.
Applying \cite{OSS} to $F_1\cup F_2$, we obtain $3\sigma(N_1\cup_{\rm id}\overline{N_2})=S(F_1\cup F_2)\cdot S(F_1\cup F_2)$.
Then, Novikov additivity of the signature together with~\Cref{Lem3.1} yields
$$
3\sigma(N_1)-S(F_1)\cdot S(F_1)=3\sigma(N_2)-S(F_2)\cdot S(F_2).
$$
Therefore, $I(g)$ is independent of the choice of an extension.
\end{proof}

If one non-singular extension exists, the invariant above is equal to three times the signature of its source manifold. 
Consequently, every other extension is constrained by the difference of signatures, as follows.

\begin{thm}\label{Thm7.7}
Let $g\colon M\times[0,1)\to L$ be a submersion such that $g|_{M\times\{0\}}$ is a Morin map from a closed oriented $3$-manifold $M\times\{0\}$ to an orientable $3$-manifold $L$. 
Suppose that $g$ admits a non-singular extension $F\colon N\to L$ of a compact oriented $4$-manifold $N$.
Then, for every extension $G\colon\widetilde{N}\to L$ of $g$ on a compact oriented $4$-manifold $\widetilde{N}$, we have 
$$
3(\sigma(\widetilde{N})-\sigma(N))=S(G)\cdot S(G).
$$
In particular, if $\sigma(N)=\sigma(\widetilde{N})$, then $S(G)\cdot S(G)=0$.
\end{thm}

\begin{proof}
By applying \Cref{Lem7.6} to $F$, we obtain $I(g)=3\sigma(N)$.
On the other hand, applying the same lemma to an extension $G$ gives $I(g)=3\sigma(\widetilde{N})-S(G)\cdot S(G)$.
Comparing these two equations, we conclude that $S(G)\cdot S(G)=3(\sigma(\widetilde{N})-\sigma(N))$.
In particular, if $\sigma(\widetilde{N})=\sigma(N)$, then $S(G)\cdot S(G)=0$. 
\end{proof}

Thus, the existence of a single non-singular extension determines a relation between the signature and the self-intersection number for every other extension of the same prescribed map.
We apply this theorem to provide a new example that admits no non-singular extension.
In the following example, denote by $K(k)$ the closed oriented $3$-manifold obtained by Dehn surgery along a knot $K\subset S^3$ with framing $k\in\mathbb{Z}$.

\begin{example}
Let $K\subset S^3$ be a knot.
We show that there exists a submersion $g\colon K(-1)\times[0,1)\to\mathbb{R}^3$ such that $g|_{K(-1)\times\{0\}}$ is a stable map and it admits no non-singular extension.
Let $N$ be the $4$-manifold obtained by attaching a single $2$-handle to $D^4$ along $K$ with framing $-1$.
Then, $\partial N$ is diffeomorphic to $K(-1)$ and $\sigma(N)=-1$.
Let $\Sigma\subset N$ be the closed orientable surface obtained by gluing a Seifert surface for $K\subset S^3$ to the core of the $2$-handle along $K$.
By~\cite{KS}, there exists a stable map $F\colon N\to\mathbb{R}^3$ such that $S(F)=\Sigma$.
Let $V$ be a collar neighborhood of $\partial N$ and set $g=F|_V$.
By the construction of $F$, the map $g$ is a submersion and its restriction to the boundary is a stable map.
Moreover, by the construction of $\Sigma$, we have $\Sigma\cdot\Sigma=-1$.
Since $\sigma(N)=-1$ and $S(F)=\Sigma$, it follows that $3\sigma(N)-S(F)\cdot S(F)=-2$.
Now suppose that $G\colon\widetilde N\to\mathbb{R}^3$ is a non-singular extension of $g$.
Applying \Cref{Thm7.7} to the non-singular extension $G$ and the extension $F$, we obtain
$$
3\sigma(N)-S(F)\cdot S(F)=3\sigma(\widetilde{N}).
$$
Thus, we conclude that $3\sigma(\widetilde N)=-2$.
Therefore, $g$ admits no non-singular extension.
\end{example}

\appendix

\section{Fold maps from $\mathbb{C}P^m$ to $\mathbb{R}^l$}

In \Cref{Fig8}, we presented results concerning the existence and non-existence of fold maps from $\mathbb{C}P^m$ to $\mathbb{R}^l$ for $1\leq m, l\leq 6$ and $2m>l$.
Here, we prove the proposition. 
We first consider the cases $l=1$ and $l=2$.

\begin{prop}
For every $m\geq 1$, $\mathbb{C}P^m$ admits a fold map into $\mathbb{R}$.
Moreover, $\mathbb{C}P^m$ admits a fold map into $\mathbb{R}^2$ if and only if $m$ is odd.
\end{prop}

\begin{proof}
The assertion for $l=1$ follows from the existence of Morse functions on $\mathbb{C}P^m$. 
We next consider the case $l=2$. 
Since $\chi(\mathbb{C}P^m)=m+1$, the Euler characteristic of $\mathbb{C}P^m$ is even if and only if $m$ is odd. 
Therefore, by \cite{Lev1}, the assertion follows. 
\end{proof}

To apply \cite[Corollary~2.4]{SSS}, we recall the following lemma by Steer~\cite{Ste}.
For the definition of a stable span, see~\cite{SSS} for example.

\begin{lem}[Steer]
For $1\leq m\leq 6$, the values of $\mathrm{span}^0(\mathbb{C}P^m)$ are given in~\Cref{Table1}, where $\mathrm{span}^0$ is the stable span of the relevant set. 
\end{lem}

\begin{table}[h]
\centering
\renewcommand{\arraystretch}{1.5} 
\setlength{\tabcolsep}{14pt}     
\begin{tabular}{c|cccccc}
\hline
$m$
& $1$ & $2$ & $3$ & $4$ & $5$ & $6$ \\
\hline
$\mathrm{span}^{0}(\mathbb{C}P^{m})$
& $2$ & $0$ & $4$ & $0$ & $2$ & $0$ \\
\hline
\end{tabular}
\caption{The stable spans of $\mathbb{C}P^{m}$ for $1\leq m\leq 6$.}
\label{Table1}
\end{table}

From \cite[Corollary~2.4]{SSS}, which is a corollary of Ando's result in~\cite{And}, it follows that $\mathrm{span}^0(\mathbb{C}P^m)\geq l-1$ guarantees that there exists a fold map $\mathbb{C}P^m\to\mathbb{R}^{l}$.
Furthermore, when $l$ is even, the converse also holds. 
Therefore, we obtain the following conclusions for $l\geq 3$.

\begin{prop}
When $m=3$, $\mathbb{C}P^3$ admits fold maps into each of
$\mathbb{R}^3$, $\mathbb{R}^4$, and $\mathbb{R}^5$.
When $m=4$ or $m=6$, $\mathbb{C}P^m$ admits no
fold map into either $\mathbb{R}^4$ or $\mathbb{R}^6$.
When $m=5$, $\mathbb{C}P^5$ admits a fold map into
$\mathbb{R}^3$, but no fold map into either $\mathbb{R}^4$ or
$\mathbb{R}^6$.
\end{prop}

For the cases $l=3$ or $l=5$, we need to use some additional results.
When $m=2$ and $l=3$, it follows from \cite{Sae1} that $\mathbb{C}P^2$ admits no fold map into $\mathbb{R}^3$. 
Moreover, when $m=4$ or $m=6$, it follows from \cite{SSS} that $\mathbb{C}P^m$ admits fold maps into $\mathbb{R}^3$. 
When $m=4$ or $m=6$, and $l=5$, it follows from \cite{KiS} that $\mathbb{C}P^m$ admits no fold map into $\mathbb{R}^5$, since $\chi(\mathbb{C}P^m)\equiv 1\pmod 2$. 
However, to the best of the author’s knowledge, it is not known whether $\mathbb{C}P^5$ admits a fold map into $\mathbb{R}^5$. 
By these arguments, we obtain~\Cref{Fig8}.

\section{Fold maps of $\mathbb{C}P^m\#\overline{\mathbb{C}P^m}$ into $\mathbb{R}^l$}
In \Cref{Fig9}, we summarized results on the existence and non-existence of fold maps from $\mathbb{C}P^m\#\overline{\mathbb{C}P^m}$ to $\mathbb{R}^l$ for $1\leq m, l\leq 6$ and $2m>l$.
We now prove these results. 
We begin with the cases $l=1$, $l=2$, and $l=3$.

\begin{prop}
For $l=1,2$, or $l=3$ and every $m\geq 1$ satisfying $2m>l$, $\mathbb{C}P^m\#\overline{\mathbb{C}P^m}$ admits a fold map into $\mathbb{R}^l$.  
\end{prop}

\begin{proof}
We first consider the case $l=1$. 
It follows from the existence of a Morse function on $\mathbb{C}P^m\#\overline{\mathbb{C}P^m}$. 
We next consider the case $l=2$. 
We have $\chi(\mathbb{C}P^m\#\overline{\mathbb{C}P^m})=2m$. 
Therefore, by \cite{Lev1}, $\mathbb{C}P^m\#\overline{\mathbb{C}P^m}$ admits a fold map into $\mathbb{R}^2$. 
Finally, we consider the case $l=3$. 
When $m=2$, it follows immediately from \cite{Sae5}.
When $m=3$, it follows from \cite[Theorem~5.9]{SSS}. 
For every $m\geq 4$, it follows from \cite[Remark~5.11]{SSS}.
\end{proof}

We next consider the case $l=4$. 
We restrict our attention to $m\geq 5$, where \cite[Theorems~4.2 and~4.4]{SSS} apply.
For $m=3$ and $4$, we do not know whether $\mathbb{C}P^m\#\overline{\mathbb{C}P^m}$ admits a fold map into $\mathbb{R}^4$.

\begin{prop}
When $m=2t$ for some $t>2$, $\mathbb{C}P^{2t}\#\overline{\mathbb{C}P^{2t}}$ admits a fold map into $\mathbb{R}^4$ if and only if $t$ is even. 
When $m=2t+1$ for some $t>1$, $\mathbb{C}P^{2t+1}\#\overline{\mathbb{C}P^{2t+1}}$ admits a fold map into $\mathbb{R}^4$ if and only if $t$ is odd. 
\end{prop}

\begin{proof}
By \cite[Theorem 4.2]{SSS}, $\mathbb{C}P^{2t}\#\overline{\mathbb{C}P^{2t}}$ admits a fold map into $\mathbb{R}^4$ if and only if $w_{4t-2}(\mathbb{C}P^{2t}\#\overline{\mathbb{C}P^{2t}})=0$. 
Let $x$ and $y$ be the canonical generators of $H^{2}(\mathbb{C}P^{2t};\mathbb{Z})$ and $H^{2}(\overline{\mathbb{C}P^{2t}};\mathbb{Z})$, respectively; see \cite{MS}.
We use the same notation for their reductions modulo $2$. 
Then, 
$$
w_{4t-2}(\mathbb{C}P^{2t}\#\overline{\mathbb{C}P^{2t}})
=
\begin{pmatrix}
2t+1 \\
2t-1
\end{pmatrix}
(x^{2t-1}+y^{2t-1}). 
$$
Since 
$$
\begin{pmatrix}
2t+1 \\
2t-1
\end{pmatrix}
\equiv t \pmod 2, 
$$
the assertion follows. 
Similarly, another assertion follows from \cite[Theorem 4.2]{SSS}.
\end{proof}

For $l=5$, we do not obtain any results on the existence or non-existence.
Finally, we consider the case $l=6$.
We focus on the cases $m=5$ and $m=6$. 
On the other hand, we have not obtained any results for $m=4$.

\begin{prop}
$\mathbb{C}P^{5}\#\overline{\mathbb{C}P^{5}}$ and $\mathbb{C}P^{6}\#\overline{\mathbb{C}P^{6}}$ do not admit a fold map into $\mathbb{R}^6$. 
\end{prop}

\begin{proof}
Suppose that $\mathbb{C}P^{5}\#\overline{\mathbb{C}P^{5}}$ admits a fold map into $\mathbb{R}^6$.
It follows from \cite[Corollary~2.4]{SSS} that there exists a real vector bundle $\xi$ of rank $5$ over $\mathbb{C}P^{5}\#\overline{\mathbb{C}P^{5}}$ such that 
$$
T(\mathbb{C}P^{5}\#\overline{\mathbb{C}P^{5}})\oplus\varepsilon^1
\cong
\xi\oplus\varepsilon^6. 
$$
where $\varepsilon^s$ is a trivial vector bundle of rank $s$.  
Taking Stiefel--Whitney classes, we obtain $w_8(\mathbb{C}P^{5}\#\overline{\mathbb{C}P^{5}})=w_8(\xi)$. 
Since $\xi$ is of rank $5$, we have $w_8(\xi)=0$. 
On the other hand, $w_8(\mathbb{C}P^{5}\#\overline{\mathbb{C}P^{5}})\neq 0$, which is a contradiction. 
The assertion for $\mathbb{C}P^{6}\#\overline{\mathbb{C}P^{6}}$ follows by a similar argument, using \cite[Corollary~2.4]{SSS} together with $w_{10}(\mathbb{C}P^{6}\#\overline{\mathbb{C}P^{6}})\neq 0$. 
\end{proof}

\end{document}